\documentclass[11 pt]{article}
\usepackage[utf8]{inputenc}
\pdfoutput=1
\usepackage{amsmath}
\usepackage{amssymb}
\usepackage{graphicx}
\usepackage{setspace}
\usepackage{centernot}
\usepackage{amsthm}
\usepackage{fullpage}
\usepackage{amsmath,amssymb,amsfonts}
\usepackage{float}
\usepackage{mathtools}
\usepackage{array}
\usepackage{multicol}
\usepackage{enumitem}
\usepackage{tcolorbox}
\usepackage{mathrsfs}
\usepackage{tikz-cd}
\usepackage{xypic}
\usepackage{geometry}
\usepackage{bbm}
\usepackage[backend=biber,style=alphabetic]{biblatex}
\usepackage{hyperref}
\hypersetup{
	colorlinks=true,
	linkcolor=blue,
	filecolor=cyan,      
	urlcolor=blue,
	citecolor=red,
}

\newtheorem{theorem}{Theorem}[section]
\newtheorem{proposition}[theorem]{Proposition}
\newtheorem{lemma}[theorem]{Lemma}

\theoremstyle{definition}
\newtheorem{definition}{Definition}
\newtheorem{claim}{Claim}[theorem]
\theoremstyle{remark}

\newtheorem{remark}[theorem]{Remark}
\theoremstyle{conjecture}
\newtheorem{conjecture}{Conjecture}[section]
\newcommand{\rG}{\mathrm G}

\newcommand{\rP}{\mathrm P}
\newcommand{\rQ}{\mathrm Q}
\newcommand{\rL}{\mathrm L}
\newcommand{\rU}{\mathrm U}
\newcommand{\rW}{\mathrm W}
\newcommand{\rT}{\mathrm T}
\newcommand{\rB}{\mathrm B}

\newcommand{\fg}{\mathfrak g}
\newcommand{\fb}{\mathfrak b}
\newcommand{\ft}{\mathfrak t}
\newcommand{\fp}{\mathfrak p}
\newcommand{\fl}{\mathfrak l}
\newcommand{\fn}{\mathfrak n}

\newcommand{\FT}{\mathrm{FT}}
\newcommand{\Res}{\mathrm{Res}}
\newcommand{\res}{\mathrm{res}}

\newcommand{\Z}{\mathbb Z}

\newcommand{\N}{\mathbb N}
\newcommand{\Q}{\mathbb Q}
\newcommand{\C}{\mathbb C}

\newcommand{\HH}{\mathcal H}

\newcommand{\M}{\mathcal{M}}
\newcommand{\X}{\mathcal{X}}
\newcommand{\ZZ}{\mathcal{Z}}
\newcommand{\mc}{\mathcal}

\title{A description of the integral depth-$r$ Bernstein center}
\author{Sarbartha Bhattacharya and Tsao-Hsien Chen }

\date{}

\begin{document}
\maketitle

\begin{abstract}
In this paper we give a description 
of the depth-$r$ Bernstein 
center $\ZZ^r(G)$ for non-negative integers $r$ of a reductive simply connected group $G$ over a non-archimedean local field $\mathbbm k$
as a limit of depth-$r$ standard parahoric Hecke algebras. Using the description, we construct maps from the algebra of stable functions 
on the  $r$-th Moy-Prasad filtration quotient of hyperspecial parahorics to 
$\ZZ^r(G)$ and use them to attach to each depth-$r$ irreducible representation $\pi$ of $G(\mathbbm k)$ an invariant $\theta(\pi)$, called the depth-$r$ Deligne-Lusztig parameters of $\pi$. We show that $\theta(\pi)$ is equal to the semi-simple part of minimal $K$-types of $\pi$. 
\end{abstract}
\tableofcontents
\section{Introduction}
\subsection{Main results}
Let $G$ be a split connected reductive simply connected algebraic group defined over a non-archimedean local field $\mathbbm k$. 
Let $R(G)$ denote the category of smooth complex representations of $G(\mathbbm k)$ and 
let $\mathcal{Z}(G)=\mathrm{End}(\mathrm{Id}_{R(G)})$ denote the Bernstein center. For any non-negative rational number $r\in\mathbb Q_{\geq0}$,
let $R(G)_{\leq r}$ (resp, $R(G)_{> r}$) denote the full subcategory of representations whose irrreducible subquotients have depth $\leq r $ (resp, depth $>r$).
Results of Bernstein and Moy-Prasad (\cite{Ber, MP94, MP96}) imply that the category $R(G)$ decomposes as a direct sum
 $R(G) = R(G)_{\leq r} \oplus R(G)_{> r}$, and hence the Bernstein center also decomposes as $\ZZ(G) = \ZZ^r(G) \oplus \ZZ^{>r}(G)$.
 We aim to provide  a description of the depth-$r$ Bernstein center $\ZZ^r(G)$.
 In this paper, we consider the case of 
integral depths. The case of rational depths will be treated in the forthcoming work.

To state the first result,  
let $Par$ be the set of standard parahoric subgroups of $G(\mathbbm k)$ 
containing a fixed Iwahori subgroup $I$.
We fix a Haar measure $\mu$ on $G(\mathbbm k)$, and for any $P\in Par$ and $r\in\mathbb Z_{\geq0}$, we denote by $P_r^+$ the $r$-th congrnence subgroup of the pro-unipotent radical $P^+$ of $P$.
We define 
$$\mathcal{M}^r_{P} := C^{\infty}_c(\frac{G(\mathbbm k)/P_r^+}{P}) $$
to be the algebra of compactly supported smooth functions on $G(\mathbbm k)$ which are $P_r^+$ bi-invariant and $P$ conjugation invariant with 
multiplication given by the convolution product with respect to the Haar measure $\mu$.
We will call $\mathcal M_P^r$, $P\in Par$,
the depth-$r$ standard parahoric Hecke algebras.
For any $P\subset Q \in Par$, we have an algebra map 
\begin{align*}
\phi^r_{P,Q}:\M^r_{Q} &\longrightarrow \M^r_{P}\\
h &\longmapsto h*\delta_{P_r^+}  
\end{align*}
 where $\delta_{P_r^+}$
 is the multiplicative unit of $\mathcal M_P^r$.
With the above defined maps, we have an inverse system $\{\M^r_{P}\}_{P \in Par}$ and we define $A^r(G)$ to be the inverse limit of the algebras $\M^r_{P}$
$$ A^r(G) := \lim_{P\in Par} \M^r_{P}$$
Our first main result is the following description of $\ZZ^r(G)$ as a limit of 
depth-$r$ standard parahoric Hecke algebras.
  
\begin{theorem}\label{main}
    There is an explicit algebra isomorphism 
    $[A^r]:A^r(G)\stackrel{\simeq}\longrightarrow\ZZ^r(G)$.
\end{theorem}
Theorem \ref{main} is restated in Theorem \ref{main thm depth r}.
We refer to Section \ref{description depthr} for a more
detailed explanation of the statement. There is an obvious evaluation map $\Psi^r:\ZZ^r(G)\to A^r(G)$
and the main step in the proof is to show that the evaluation map is an isomorphism by constructing an explicit inverse map $[A^r]=(\Psi^r)^{-1}$. 
The construction of $[A^r]$ is a generalization of the work of Bezrukavnikov-Kazhdan-Varshavsky in \cite{BKV,BKV2}
where they gave an explicit construction of the depth-$r$ Bernstein projector. 
A key technical result is a stabilization property of certain averaging maps, see Theorem \ref{thm1r}.

To state the second result,
let $P\in Par$ be a hyperspecial parahoric subgroup and consider the $r$-th Moy-Prasad filtration quotient $P_r/P_{r}^+$ of $P$, where $P_r$ is  the $r$-th congruence subgroup of $P$.
Note that $P_0/P_{0}^+$ is isomorphic to
a connected finite reductive group, and 
$P_r/P_{r}^+$, $r>0$, is (non-canonically) isomorphic to its adjoint representation.
Inspired by the work of 
Laumon-Letellier \cite{laumon}
and the first author \cite{chen,chen2}, we  introduce and study the 
algebra $C^{st}(P_r/P_r^+)$ of stable functions on $P_r/P_r^+$ in Section \ref{stable functions}.

\begin{theorem}\label{main2}
There is an algebra homomorphism $\xi^r:C^{st}(P_r/P_r^+)\to \ZZ^r(G)$.
\end{theorem}

Theorem \ref{main2} is a combination of Theorem \ref{stable group to center} 
and Theorem \ref{stable lie to center}  in the paper.
We refer to Section \ref{stable to center} for a more
detailed explanation of the statement. 
The proof of Theorem \ref{main2} uses the description of the depth-$r$ Bernstein center in Theorem \ref{main} and also 
a vanishing property of stable functions, see  
Theorem \ref{vanishing group} and Theorem \ref{main thm stable Lie}.

As an application of Theorem \ref{main2}
we attach to each depth-$r$ irreducible representation 
an invariant,  called the 
depth-$r$ Deligne-Lusztig parameter,
 with remarkable properties:
 \begin{theorem}\label{main3}
    One can attach to each depth-$r$ irreducible representation $(\pi,V)$ of $G(\mathbbm k)$
    an element  
$\theta(\pi)$ in
$(\hat\rT//\rW)^F$
if $r=0$
or in $(\ft_r//\rW)^F$ if $r>0$, to be called
the depth-$r$ Deligne-Lusztig parameter of $\pi$,
characterized by the property that 
the composed map 
$C^{st}(P_r/P_r^+)\stackrel{\xi_r}\to \ZZ^r(G)\to \mathrm{End}
(\pi)\cong\mathbb C$ is given by evaluation at 
$\theta(\pi)$.
Moreover, $\theta(\pi)$ is equal to the semi-simple part of minimal 
$K$-types of $\pi$.
 \end{theorem}
Theorem \ref{main3} is a combination of results in
Section \ref{DL} and 
Theorem \ref{DL=K-types}. 
We refer to Section \ref{DL} and \ref{K-types} for a more
detailed explanation of the statement.  
In \emph{loc. cit.}, we also show that 
 the depth-$r$ Deligne-Lusztig parameters $\theta(\pi)$ agree
with the depth-$r$ restricted Langlands parameter
of $\pi$ introduced in \cite{chendebackertsai}, see Remark \ref{remark CDT}.

The main results of the paper 
allow the possibility 
of studying Bernstein centers of $p$-adic groups using the 
theory of \'etale cohomology and perverse sheaves. Theorem \ref{main2} provides such an example  where 
we apply Deligne-Lusztig theory \cite{DL76} and Fourier transforms on Lie algebras \cite{Le,Lu1}
to construct elements in the depth-$r$ Bernstein center that ``see" the 
minimal $K$-types. 
In joint work with Charlotte Chan, we plan to use the depth-$r$ character sheaves  studied in \cite{BC,Lu2,Lu3}
to construct and study more elements in the Bernstein center.

\subsection{Rational depths}
In our later work \cite{bcrationaldepths}, we generalize the main results of the paper to general rational depths $r\in\mathbb Q_{\geq0}$. 
In particular, we obtain  a decomposition of the
category of smooth representations into a product of full subcategories indexed by restricted
depth-$r$ Langlands parameters.
A new ingredient in the rational depths case is the study of 
invariant functions on  finite graded Lie algebras (in the integral depths  case we only need  adjoint invariant functions on finite reductive groups or  Lie algebras). 

\subsection{A conjecture on stability}
We conclude the introduction with the following conjecture on stability. 
Recall that 
an element $z\in\ZZ(G)$ in the Bernstein center  is called stable if the associated 
invariant distribution $\nu_z$
on $G(\mathbbm k)$ is stable.
Let $\mathcal Z^{st}(G)\subset\mathcal Z(G)$ be the subspace of stable elements
and $\mathcal Z^{st,r}(G)=\mathcal Z^{r}(G)\cap \mathcal Z^{st}(G)$.
\begin{conjecture}\label{stability}
We have $\mathrm{Im}(\xi^r)\subset\mathcal Z^{st,r}(G)$.
\end{conjecture}

Let
$\mathrm 1_e\in C^{st}(P_r/P_{r}^+)$ be the multiplicative unit.
The image $z^r=\xi^r(\mathrm 1_e)\subset \ZZ^r(G)$
is the so-called depth-$r$ Bernstein projector
and its stability was proved in
\cite[Theorem 1.23]{BKV2}.
 This example provides an evidence of 
 Conjecture \ref{stability}.
 It is work in progress to prove
 Conjecture \ref{stability} in the $r=0$ case using  the geometric approach to the depth-zero stable center conjecture in \cite{chen2,BKV}.

 \begin{remark}
     Conjecture \ref{stability} is suggested by the conjectural theory of $L$-packets 
and its relation to endoscopy of invariant distributions. Namely, it is expected that 
an element $z\in\ZZ(G)$ is stable if and only if the corresponding function 
on the set  of equivalence classes of irreducible representations of $G(\mathbbm k)$ is constant on $L$-packets and 
the set of irreducible representations of the same depth-$r$ Deligne-Lusztig parameter is a union of $L$-packets.
 \end{remark}

 \subsection{Organization}
We briefly summarize the main goals 
of each section. In Section \ref{depth zero center}
we give a description of the depth-zero Bernstein center.
In Section \ref{description depthr}, 
we give a description of the depth-$r$ Bernstein center for $r \in \Z_{>0}$. 
In section \ref{stable functions}, we introduce and study stable functions 
on finite reductive groups and Lie algebras. 
In Section \ref{stable to center}, we construct 
maps from the algebras of stable functions on the 
 $r$-th Moy-Prasad filtration quotient  of hyperspecial parahorics
to the depth-$r$ Bernstein center.
We introduce the notion of 
depth-$r$ Deligne-Lusztig parameters and study their relationship to minimal $K$-types.

\subsection{Acknowledgement}
 The authors thank 
 Roman Bezrukavnikov,
 Charlotte Chan, Bao Ch\^au Ng\^o,
 Cheng-Chiang Tsai, Zhiwei Yun
 for many useful discussions. T.-H. Chen also thanks the 
 NCTS-National Center for Theoretical Sciences at Taipei
 where parts of this work were
done. The research of T.-H. Chen is supported by NSF grant DMS-2143722.

\section{A description of the depth-zero Bernstein center}\label{depth zero center}
Let $Irr(G)$ denote the set of equivalence classes of irreducible representations of $G(\mathbbm k)$. 
For  $r\in\mathbb Q_{\geq0}$, let $Irr(G)_{\leq r}$ (resp. $Irr(G)_{> r}$) denote the set of $(\pi, V) \in Irr(G)$ of depth $\leq r$ (resp. $>r$) and 
let $R(G)_{\leq r}$ (resp, $R(G)_{> r}$) denote the full subcategory of representations whose irrreducible subquotients have depth $\leq r $ (resp, depth $>r$).

Let $\mathcal H(G):=(C^{\infty}_c(G),*)$ be the Hecke algebra  of smooth compactly supported functions on $G(\mathbbm k)$ with multiplication given by  the convolution product 
 \[h* h'(x)=\int_{} h(xy^{-1})h'(y)d\mu(y)\]
 with respect to the fixed Haar measure $\mu$. Let $\mc X(G) = \X$ denote the (reduced) Bruhat-Tits building of the group $G(\mathbbm k)$, and $G_x$ denote the parahoric subgroup corresponding to $x\in \X$. Further, let $(G_x)_r$ and $(G_x)_r^+$ denote the Moy-Prasad filtration subgroups as defined in \cite{MP94}, although our notation is slightly different (\cite{MP94} uses $P_{x,r}$ and $P_{x,r+}$ for these respectively). \par
 We give a brief description of some equivalent notions of the Bernstein center, mainly following \cite{Ber,BKV}. Given $(\pi, V) \in R(G)$, each $z\in \ZZ(G)$ defines an endomorphism $z|_V \in \text{End}_{G(\mathbbm k)}(V)$. In particular, if $(\pi, V) \in Irr(G)$, each $z\in \ZZ(G)$ defines a function $f_z :Irr(G)  \rightarrow \C$ such that $z|_V=f_z(\pi)Id_V$. Moreover, the map $z\mapsto f_z$ is an algebra homomorphism $\ZZ(G) \rightarrow Fun(\: Irr(G), \: \C)$, which is injective. 
 
Each $z\in \ZZ(G)$ defines an endomorphism $z_{reg}$ of the $G(\mathbbm k)$ representation on $\HH(G)$ given by the conjugation action ($(g f) (x) = f(g^{-1}xg))$, and hence gives rise to an $G(\mathbbm k)$-invariant distribution $\nu_z$ such that $\nu_z(f)=z_{reg}(i^*(f))(1) $ for all $f\in \HH(G)$, where $i:G(\mathbbm k)\rightarrow G(\mathbbm k)$ is given by $g \mapsto g^{-1}$. The invariant distribution $\nu_z$ can be characterised by the condition $\nu_z * h = z_\HH(h) \: \forall \: h\in \HH(G)$. Moreover, the map $z \mapsto \nu_z$ gives an isomorphism of $\ZZ(G)$ onto the algebra of essentially compact $G(\mathbbm k)$-invariant distributions $\mc D(G)^G_{ec}$.

 \par
  Each smooth $G(\mathbbm k)$ representation is equivalently a non-degenerate $\HH(G)$-module. Let $(l, \HH(G))$ and $(r,\HH(G))$ denote the $G(\mathbbm k)$ representations induced by left and right translations by $G(\mathbbm k)$ on $\HH(G)$. The action on $G(\mathbbm k)$ by $G(\mathbbm k)^2$, defined by $(g,h)(x)=gxh^{-1}$ gives a $G(\mathbbm k)^2$ action on $\HH(G)$, given by $(g,h)f(x)=l(g)r(h)f(x)= f(g^{-1}xh)$, and hence $\HH(G)^2$-module structure on $\HH(G)$. Note that the actions $l$ and $r$ commute, and the action of $\HH(G)^2$ on $\HH(G)$ is given by $(\alpha, \beta)f= \alpha * f* \hat{\beta}$, where $\hat{\beta}(x)=\beta(x^{-1})$. Each $z\in \ZZ(G)$ defines an endomorphism $z_\HH$ of the smooth representation $(l, \HH(G))$, and since the actions $l$ and $r$ commute, the endomorphism $z_\HH$ of the  Hecke algebra $\HH(G)$ commutes with left and right $G(\mathbbm k)$ actions and hence left and right convolutions. For every $(\pi, V) \in R(G)$, $v\in V$ and $h\in \HH(G)$, we have the equality $z_V(h(v))=(z_\HH(h))(v)$. Moreover, the map $z\mapsto z_\HH$ defines an algebra isomorphism $\ZZ(G) \xrightarrow{\sim} \text{End}_{\HH(G)^{2}}(\HH(G))$. In this paper, we have used this description of the Bernstein center to produce algebra isomorphisms onto the depth-$r$ parts for each non-negative integer $r$, and finally a limit description of the entire center. \par
 
\subsection{Stabilization in the depth-zero case}
 Let $Par$ be the set of standard parahorics for $G(\mathbbm k)$ containing a fixed Iwahori subgroup $I$ and $P\in \:Par$. We define 
$$\mathcal{M}^0_P := C^{\infty}_c(\frac{G(\mathbbm k)/P^+}{P}) $$
to be the subalgebra of 
$\HH(G)$ consisting of compactly supported smooth functions on $G(\mathbbm k)$ which are $P^+$ bi-invariant and $P$ conjugation invariant. 
For $P,Q \in Par$ and $P\subseteq Q$, we have a map 
\begin{align*}
\phi^0_{P,Q}:\M^0_Q &\longrightarrow \M^0_P\\
h &\longmapsto h*\delta_{P^+}  
\end{align*}
where $\delta_{K} = \frac{1}{\mu(K)}\mathbbm{1}_K$ for any $K \subseteq G(\mathbbm k)$, $\mathbbm{1} _K$ being the characteristic function of $K$ and $\mu(K)=\int \mathbbm{1} _Kd\mu$.
With the above defined maps, we have an inverse system $\{\M^0_P\}_{P \in Par}$ and we define $A^0(G)$ to be the inverse limit of the algebras $\M^0_P$.
$$ A^0(G) := \lim_{P\in Par} \M^0_P.$$

Let $\Omega$ be the set of all finite $I$-invariant subsets $Y\subset G(\mathbbm k)/I$ such that for all $w'\leq w$, the image of $Iw' \subset Y$ if the image of $Iw \subset Y$. For every $Y \in \Omega$ and $P \in Par $, we denote the image of $Y$ in $G(\mathbbm k) /P$ by $Y_P$. For any $h_P\in \M^0_P$, we define $Av^{Y_P}(h_P) \in \mc H(G)$ to be the function 
\begin{equation}
    Av^{Y_P}(h_P)= \sum_{y\in Y_P}Ad_y(h_P).
\end{equation}
Note that this is well defined since $h_P$ is $P$-conjugation invariant. Let $\tilde{\Delta}$ denote the set of affine simple roots. There exists a bijection between proper subsets $J\subset \tilde{\Delta}$ and $Par$, and we denote the standard parahoric subgroup corresponding to $J$ by $P_J$. Let $r(G)= |\tilde{\Delta}|-1$ denote the rank of $G$ and $r(P_J)=|J|$ denote the semisimple rank of the reductive quotient. For each $Y\in \Omega $ and $h=\{h_P\}_{P\in Par}\in A^0(G)$, define $[A^Y_h] \in \mc H(G) $ as 
\begin{equation}\label{AYh}
    [A^Y_h] = \sum_{P \in Par} (-1)^{r(G)-r(P)}Av^{Y_P}(h_P).
\end{equation}

We have the following key stabilization property.
   \begin{theorem}\label{thmstab0}
       For every $f\in \mc H(G)$ and $h\in A^0(G)$, the sequence $\{[A^Y_h] * f\}_{Y \in \Omega} $ stabilizes, and hence $\lim_{Y\in \Omega} \:[A^Y_h] * f $ is well-defined.
\end{theorem}

In order to complete the proof of Theorem \ref{thmstab0} we need some notations and lemmas. Let $\Tilde{W}$ denote the extended affine Weyl group (which is the same as the affine Weyl group in this case). We define 
\begin{equation}\label{defnSY}
    S(P_n^+)=\{w\in \Tilde{W} \:|\: U_{w(\alpha)} \nsubseteq P_n^+  \: \forall\alpha \in \Tilde{\Delta}\} \cup \{1\},  \:\:\: \text{and}\:\: Y(P_n^+)= \bigcup_{w\in S(P_n^+)} Y^{\leq w}
\end{equation}
where $Y^{\leq w} =\cup_{w'\leq w} Iw'I/I \in \Omega$. For $ w\in \Tilde{W}$
\[
J_w=\{ \alpha \in \Tilde{\Delta}\: | \: w(\alpha) >0\}, \:\:\: Y_w=IwI/I
\]
\begin{lemma} \label{lemfin}
Let $P\in\: Par $
    \begin{itemize}
        \item[(a)] For every $n\geq 0$, $S(P^+_n)$ is finite. 
        \item[(b)] $S(P^+)=\{1\}$ 
    \end{itemize}
\end{lemma}
\begin{proof}
    This is Lemma $4.2.2$ of \cite{BKV}
\end{proof}

\begin{lemma}\label{lemdel}
    Let $w\in \Tilde{W}, \: \alpha \in \Tilde{\Delta}, \: Q\in \:Par $ and $n\in \mathbb{N}$. Let $J \subset J_w \backslash  \alpha $ be such that $U_{w(\alpha)} \subset Q^+_n, \: J \neq \Tilde{\Delta} \backslash \alpha $ and $J' = J \cup \{\alpha \}$. Then 
    \begin{equation}
        Av^{(Y_w)_{P_{J'}}}(\delta_{P_{J'}^+}) \:*\: \delta_{Q^+_n} = Av^{(Y_w)_{P_{J}}}(\delta_{P_{J}^+}) \:* \:\delta_{Q^+_n}
    \end{equation}
\end{lemma}
\begin{proof}
    We will prove this lemma by proving a series of claims. 
    \begin{claim}\label{cl1lemdel}
        $P_J^+ = P_{J'}^+ \:( P_J^+ \cap w^{-1}Q^+_n w)$
    \end{claim}
    It suffices to show that for every $\beta \in \Tilde{\Phi}$ and $U_\beta \subset P_J^+\backslash P_{J'}^+$, we have $U_\beta \subset w^{-1}Q^+_n w$ or equivalently $U_{w(\beta)} = w U_\beta w^{-1} \subset Q^+_n$. \par 
    For $\beta \in \Tilde{\Phi}$, $U_\beta \in P_J^+ \Longleftrightarrow \beta = \sum_{\alpha_i \in \Tilde{\Delta}} n_i \alpha_i$ where $n_i\geq 0 \: \forall i $ and $n_i >0$ for some $\alpha_i \notin J$. (Explanation : Let $y\in \X$ be such that $P= G_y$. Then $ U_\beta \in P^+ \Longleftrightarrow \beta(y) > 0$). So, any $\beta$ such that $U_\beta \subset P_J^+ \backslash P_{J'}^+$ has the form $\beta = \sum_{\alpha_i \in J} n_i\alpha_i + n'\alpha $ where $ n_i\geq 0, \: n'>0$. \par
    Let $x\in \X$ be such that $Q= G_x$. Note that $U_{w(\alpha)} \subset Q_n^+ \Longleftrightarrow w(\alpha)(x) >n $ and $w(\beta)(x) = \sum n_i w(\alpha_i)(x) +n' w(\alpha)(x) \geq n'w(\alpha)(x) \geq w(\alpha)(x)>n$. This is because $w(\alpha_i)>0\: \forall \:\alpha_i \in J \:(J \subset J_w)$. So, we have $w(\beta)(x) >n $ which gives us $U_{w(\beta)} \subset Q_n^+$, and finishes the proof of the claim. \par
    Since we have assumed that $G$ is simply connected, the extended affine Weyl group $\Tilde{W}$ is the same as the affine Weyl group and is a coxeter group. For $w \in \Tilde{W}$, let $N(w) =\{ \alpha \in \Tilde{\Phi}^+ \:|\: w(\alpha) < 0\}$, and $W_J \subset \Tilde{W}$ be the subgroup generated by the reflections corresponding to $J\subset \Tilde{\Delta}$. Consider the right cosets $\Tilde{W}/W_J$. If $J\subset J_w$, then $N(w) \cap J_w = \emptyset$ which implies that $w$ is the smallest element in the coset $w W_J$. For a standard parahoric $P_J \in Par $, $W_{P_J}=W_J$ ($W_{P_J}$ defined as in Section 3 of \cite{lansky}) So, Corollary 3.4 and Theorem 5.2 in \cite{lansky} gives us 
    \[
    \left| IwI/I\right|=\left|IwP_J/P_J\right| = q^{l(w)}
    \]
    for $J\subset J_w$, where $q$ is the size of the residue field of the local field $\mathbbm k$. 
    Hence, if we fix a set $I_w \subset I$ such that $I_w w $ forms a set of representatives of $Y_w=IwI/I$. Then by the previous assertion, it also forms a set of representatives for $(Y_w)_{P_J}= IwP_J/P_J$ for $J \subset J_w$ and 
\begin{equation}
     Av^{I_w w}(h_{P_J})= Av^{(Y_w)_{P_J}}(h_{P_J}) \text{ for } h_{P_J} \in \M_{P_J}, \: J\subset J_w
\end{equation}
   \begin{claim}\label{cl2lemdel}
      $ Av^{(Y_w)_{P_{J}}}(\delta_{P_{J}^+}) \:* \delta_{Q^+_n}= Av^{I_w w}(\delta_{P_{J}^+}\:*\:\delta_{w^{-1}Q_n^+ w} )$ if $J\subset J_w$
   \end{claim}
   \begin{align*}
       Av^{I_w w}(\delta_{P_{J}^+}\:*\:\delta_{w^{-1}Q_n^+ w} ) &= \sum_{y\in I_w w}Ad_y(\delta_{P_{J}^+}\:*\:\delta_{w^{-1}Q_n^+ w})\\
       &= \sum_{y\in I_w w}Ad_y(\delta_{P_{J}^+})\:*\:Ad_y(\delta_{w^{-1}Q_n^+ w})\\
       &= \sum_{y\in I_w }Ad_y(\delta_{w P_{J}^+ w^{-1}})\:*\:Ad_y(\delta_{Q_n^+})
 \end{align*}
 Since $I_w \subset I \subset Q$ and $Q$ normalises $Q_n^+$, we have $Ad_y(\delta_{Q_n^+})= \delta_{Q_n^+}$ for $y\in I_w$, which gives us  
 \begin{align*}
     Av^{I_w w}(\delta_{P_{J}^+}\:*\:\delta_{w^{-1}Q_n^+ w} ) &= \sum_{y\in I_w }Ad_y(\delta_{w P_{J}^+ w^{-1}})\:*\:Ad_y(\delta_{Q_n^+}) \\
     &= \sum_{y\in I_w }Ad_y(\delta_{w P_{J}^+ w^{-1}})\:*\:\delta_{Q_n^+}\\
     &= Av^{(Y_w)_{P_{J}}} (\delta_{P_{J}^+})\:*\:\delta_{Q_n^+}
 \end{align*}
 which finishes the proof of the claim. Since $J,\:J'\subset J_w$, the above claim holds true for both $J$ and $J'$. To finish the proof of the Lemma, we just need to show the following statement.
 \begin{claim}\label{cl3lemdel}
     $\delta_{P_{J'}^+}\:*\:\delta_{w^{-1}Q_n^+ w} = \delta_{P_{J}^+}\:*\:\delta_{w^{-1}Q_n^+ w}$
 \end{claim}
 Computing the integrals, we get $\delta_{P_{J'}^+}\:*\:\delta_{w^{-1}Q_n^+ w} = \delta_{P_{J'}^+ w^{-1}Q_n^+ w}$ and the same is true for $J$. Using Claim \ref{cl1lemdel}, we see $P_{J'}^+ w^{-1}Q_n^+ w= P_{J}^+ w^{-1}Q_n^+ w$, which proves our assertion and finishes the proof of the lemma.
\end{proof}
The next lemma generalises the previous one  to an arbitrary $h\in A^0 (G)$.
\begin{lemma}\label{lemh}
    Let $w\in \Tilde{W}, \: \alpha \in \Tilde{\Delta}, \: Q\in \:Par $ and $n\in \mathbb{N}$. Let $J \subset J_w \backslash  \alpha $ be such that $U_{w(\alpha)} \subset Q^+_n, \: J \neq \Tilde{\Delta} \backslash \alpha $ and $J' = J \cup \{\alpha \}$. Let $h=\{ h_P\}_{P\in Par} \in A^0 (G)$. Then 
    \begin{equation}\label{eqlemh}
        Av^{(Y_w)_{P_{J'}}}(h_{P_{J'}}) \:*\: \delta_{Q^+_n} = Av^{(Y_w)_{P_{J}}}(h_{P_{J}}) \:* \:\delta_{Q^+_n}
    \end{equation}
\end{lemma}
\begin{proof}
    Using the same idea as in Lemma \ref{lemdel}, we see that 
$$
     Av^{(Y_w)_{P_{J}}}(h_{P_{J}}) \:*\: \delta_{Q^+_n} = Av^{I_w w}(h_{P_{J}}\:*\:\delta_{w^{-1}Q_n^+ w} )
$$
and the same holds true for $J'$. Thus, following the arguments in the proof of Lemma \ref{lemdel} it suffices to show that 
\[
h_{P_{J'}}\:*\:\delta_{w^{-1}Q_n^+ w} = h_{P_{J}}\:*\:\delta_{w^{-1}Q_n^+ w}
\]
From Claim \ref{cl3lemdel} of Lemma \ref{lemdel}, we have $\delta_{P_{J'}^+}\:*\:\delta_{w^{-1}Q_n^+ w} = \delta_{P_{J}^+}\:*\:\delta_{w^{-1}Q_n^+ w}$. Using that,
\begin{align*}
    h_{P_{J'}}\:*\:\delta_{w^{-1}Q_n^+ w} &= h_{P_{J'}}\:*\:(\delta_{P_{J'}^+}\: * \: \delta_{w^{-1}Q_n^+ w})\\
    &= (h_{P_{J'}}\:*\:\delta_{P_{J}^+})\: * \: \delta_{w^{-1}Q_n^+ w}\\
    &= h_{P_{J}}\:*\:\delta_{w^{-1}Q_n^+ w}.
\end{align*}
So, we are done.
\end{proof}
The next proposition is the main step towards the proof of Theorem \ref{thmstab0}. 
\begin{proposition}\label{stab}
    Let $Q\in Par, \: Y\in \Omega $. If $Y\supset Y(Q_n^+)$, we have 
    \begin{equation}
        [A^Y_h] \:*\: \delta_{Q_n^+} = [A^{Y(Q_n^+)}_h]\:*\: \delta_{Q_n^+}
    \end{equation}
\end{proposition}
\begin{proof}
    Idea of the proof: We proceed by induction on the number $I-$orbits in $Y\backslash Y(Q_n^+)$. \par
    For $Y\in \Omega$, let $w\in \Tilde{W} \backslash S(Q_n^+)$ such that $Y_w\subset Y$ is maximal and hence open. Note that given the conditions on the elements in $\Omega$, if we chose $Y_w$ in the above-mentioned way, then $Y'=Y\backslash Y_w \in \Omega $, and we can use induction. 
    \begin{claim}\label{cl1stab}
       $[A^Y_h] \:*\: \delta_{Q_n^+} = [A^{Y'}_h]\:*\: \delta_{Q_n^+}$ 
    \end{claim}
    Consider 
    \begin{equation}\label{Awh}
        [A^w_h]= \sum_{J\subset J_w} (-1)^{r(G)-|J|} Av^{(Y_w)_{P_J}}(h_{P_J}).
    \end{equation}

    Note that since $Y_w \subset Y$ is maximal (and hence open), we have 
    \[
    Y_{P_J}\backslash Y_{P_J}'=
    \begin{cases}
        (Y_w)_{P_J} &J\subset J_w,\\
        \emptyset &\text{otherwise}.
    \end{cases}
    \]
\begin{equation}   
  \begin{aligned}
    [A^{Y'}_h]&=\sum_{J\subset J_w} (-1)^{r(G)- r(P_J)} Av^{Y_{P_J}'}(h_{P_J}) +  \sum_{J\nsubseteq J_w} (-1)^{r(G)- r(P_J)} Av^{Y_{P_J}'}(h_{P_J})\\
    &=\sum_{J\subset J_w}(-1)^{r(G)- r(P_J)} Av^{Y_{P_J}'}(h_{P_J}) +  \sum_{J\nsubseteq J_w} (-1)^{r(G)- r(P_J)} Av^{Y_{P_J}}(h_{P_J})
  \end{aligned}
  \end{equation}
since $Y_{P_J} = Y_{P_J}'$, when $J\nsubseteq J_w$. So, using the definition of $[A^Y_h]$ and \eqref{Awh}, we see that 
\begin{equation}
    \begin{aligned}
        [A^{Y}_h]&=\sum_{J\subset J_w} (-1)^{r(G)- r(P_J)} Av^{Y_{P_J}}(h_{P_J}) +  \sum_{J\nsubseteq J_w} (-1)^{r(G)- r(P_J)} Av^{Y_{P_J}}(h_{P_J})\\
        &=\sum_{J\subset J_w} (-1)^{r(G)- r(P_J)} (Av^{(Y_w)_{P_J}}(h_{P_J}) +Av^{Y_{P_J}'}(h_{P_J})) +  \sum_{J\nsubseteq J_w} (-1)^{r(G)- r(P_J)} Av^{Y_{P_J}'}(h_{P_J})\\
        &= [A^{Y'}_h] + [A^w_h]
    \end{aligned}
\end{equation}

In order to prove the claim, it is enough to show that $[A^w_h]\:*\: \delta_{Q_n^+}=0 \:\forall w \in \Tilde{W} \backslash S(Q_n^+)$ such that $Y_w \subset Y$. By definition of $S(Q_n^+)$, for each $w \in \Tilde{W} \backslash S(Q_n^+)$, $\exists \alpha \in \Tilde{\Delta}$ such that $U_{w(\alpha)} \subset Q_n^+$, and hence $\alpha \in J_w$. Let $J'=J\cup \alpha$. Using Lemma \ref{lemh}, we see
\begin{align*}
[A^w_h]\: * \: \delta_{Q_n^+} &= \sum_{J\subset J_w} (-1)^{r(G)-|J|} \left(Av^{(Y_w)_{P_J}}(h_{P_J}) \: * \: \delta_{Q_n^+} \right) \\
&= \sum_{J\subset J_w\backslash \alpha } (-1)^{r(G)-|J|} \left(Av^{(Y_w)_{P_J}}(h_{P_J})  *\delta_{Q_n^+}\right) + \sum_{J\subset J_w\backslash \alpha } (-1)^{r(G)-|J|-1} \left(Av^{(Y_w)_{P_{J'}}}(h_{P_{J'}})  *  \delta_{Q_n^+}\right)\\
&= \sum_{J\subset J_w\backslash \alpha } (-1)^{r(G)-|J|} \left(\left(Av^{(Y_w)_{P_J}}(h_{P_J}) \: * \: \delta_{Q_n^+}\right)- \left(Av^{(Y_w)_{P_{J'}}}(h_{P_{J'}}) \: * \: \delta_{Q_n^+}\right) \right)\\
&=0
\end{align*}
We have proved our claim, and hence by induction we have for $ Y(Q_n^+) \subset Y \in \Omega$,
\[
[A^Y_h] \:*\: \delta_{Q_n^+} = [A^{Y(Q_n^+)}_h]\:*\: \delta_{Q_n^+}
\]
which finishes the proof of the proposition. 
\end{proof}

Now, using the lemmas and the propositions we have stated, we can prove Theorem \ref{thmstab0}. 
\begin{proof}[Proof of Theorem \ref{thmstab0}]
    We are trying to prove that for every $f\in \mc H(G)$ and $h\in A^0(G)$, the sequence $\{[A^Y_h] * f\}_{Y \in \Omega} $ stabilizes. For any $f\in \mc H(G)$, $\exists \: n\in \N $ such that $f$ is left $I_n^+$-invariant, i.e., $f= \delta_{I_n^+} * f$. So,
    \[
    [A^Y_h] * f = [A^Y_h] * \delta_{I_n^+} * f
    \]
    Now, from Proposition \ref{stab}, we observe that for given $n \in \N $, 
    \[
    [A^Y_h] \:*\: \delta_{I_n^+} = [A^{Y(I_n^+)}_h]\:*\: \delta_{I_n^+}
    \]
    for large enough $Y \in \Omega $ such that $Y \supset Y(I_n^+)$, since $Y(I_n^+)$ is finite by Lemma \ref{lemfin}. Hence, 
    \begin{align*}
        [A^Y_h] * f &= [A^Y_h] * \delta_{I_n^+} * f\\
        &=[A^{Y(I_n^+)}_h]\:*\: \delta_{I_n^+} *f\\
        &=[A^{Y(I_n^+)}_h] *f
    \end{align*}
    for $Y \supset Y(I_n^+)$ and this finishes the proof. 
\end{proof}
\subsection{A limit description of the depth-zero Bernstein center}
As a consequence of Theorem \ref{thmstab0},  we can define $[A_h] \in \text{End}_{\HH(G)^{op}}(\HH(G))$ for each $h\in A^0 (G)$ by the formula         
\begin{equation}
    [A_h](f):=\lim_{Y\in \Omega} \:[A^Y_h] * f.
\end{equation}
\begin{proposition}\label{eval}
    Given $\{h_P\}_{P\in Par} = h\in A^0(G),\: Y\in \Omega $ and $P\in Par$, we have 
    \begin{equation}
        [A^Y_h] * \delta_{P^+} = h_P
    \end{equation}
\end{proposition}
\begin{proof}
    From Lemma \ref{lemfin}, we have that $S(P^+) = \{1\}$ and hence $Y(P^+)=Y_1$, $J_1=\Tilde{\Delta}$.  Since $Y\subset Y_1 \: \forall\: Y\in \Omega$, we have from Proposition \ref{stab} that $\forall\: Y\in \Omega$
    \[
    [A^Y_h] * \delta_{P^+} = [A^{Y_1}_h] * \delta_{P^+}
    \]
    Let $\alpha \in \Tilde{\Delta}$ be such that $U_{\alpha} \subset P^+$, $J'=J\cup \{\alpha\}$  and denote $P_{\{\Tilde{\Delta}\backslash \alpha\}} \in Par$ by $P'$. 
    \begin{align*}
        [A^{Y_1}_h] &= \sum_{P \in Par} (-1)^{r(G)-r(P)}Av^{{(Y_1})_P}(h_P)\\
        &= Av^{(Y_1)_{P'}}(h_{P'}) +\sum_{J\subsetneq \Tilde{\Delta}\backslash \alpha} (-1)^{r(G)-|J|} \left(Av^{(Y_1)_{P_J}}(h_{P_J}) - Av^{(Y_1)_{P_{J'}}}(h_{P_{J'}}) \right)\\
        &= h_{P'} + \sum_{J\subsetneq \Tilde{\Delta}\backslash \alpha} (-1)^{r(G)-|J|} (h_{P_J}-h_{P_{J'}})
    \end{align*}
Let $x\in \X$ be such that $P= G_x$. Observe that $U_\alpha \subset P^+ \Rightarrow \alpha(x)>0$. So, $P \subset P'$ and $h_{P'} * \delta_{P^+}=h_P$. If we can prove that $h_{P_J} * \delta_{P^+} = h_{P_{J'}} * \delta_{P^+}$, we are done. Note that   $h_{P_J} = h_{P_{J'}} * \delta_{P_J^+}$ since $P_{J} \subset P_{J'}$, which gives us
\begin{align*}
    h_{P_J} * \delta_{P^+} - h_{P_{J'}} *\delta_{P^+} &= h_{P_{J'}} * \delta_{P_J^+}* \delta_{P^+}-h_{P_{J'}} *\delta_{P_{J'}^+} *\delta_{P^+}\\
    &=h_{P_{J'}} * (\delta_{P_J^+}* \delta_{P^+}- \delta_{P_{J'}^+} *\delta_{P^+})
\end{align*}
meaning that it is enough to show $\delta_{P_J^+}* \delta_{P^+} = \delta_{P_{J'}^+}* \delta_{P^+}$ to complete the proof of the proposition.
\begin{claim}\label{cl1eval}
    $\delta_{P_J^+} * \delta_{P^+} = \delta_{P_{J'}^+} * \delta_{P^+}$
\end{claim}
Let $x,\: y,\: z \in \X$ be such that $P_J =G_x,\: P_{J'} = G_y $ and $P=G_z$. To prove the claim it is enough to show that $P_J^+ \cdot P^+ = P_{J'}^+ \cdot P^+$. We already know that $P_J^+ \cdot P^+ \supset P_{J'}^+ \cdot P^+$. To show the reverse inclusion, let $\beta \in \tilde{\Phi} $ be such that $U_\beta \subset P_J^+ \backslash P_{J'}^+ $. Using arguments similar to claim \ref{cl1lemdel} in the proof of Lemma \ref{lemdel}, we can conclude that $\beta$ has the form $\beta = \sum_{\alpha_i \in J} n_i\alpha_i + n'\alpha $ where $ n_i\geq 0, \: n'>0$, since $\beta(x) >0$ and $\beta(y) \leq 0$. Now, since $U_\alpha \subset P^+= G_z^+$, we have that $\alpha(z) >0$, which means that $\beta(z) >0 \Rightarrow U_\beta \subset P^+$ and finishes the proof of the claim. Note that what we have essentially proved is that for $\beta \in \tilde{\Delta}$, $\beta(x) >0 $ implies that either $\beta(y) >0 $ or $\beta(z) > 0 $. \par 
 Using the above facts, we have  
    \begin{align*}
      [A^Y_h] * \delta_{P^+} &= [A^{Y_1}_h] * \delta_{P^+}  \\
      &= h_{P'} * \delta_{P^+} + \sum_{J\subsetneq \Tilde{\Delta}\backslash \alpha} (-1)^{r(G)-|J|} (h_{P_J}* \delta_{P^+}-h_{P_{J'}}* \delta_{P^+})\\
      &= h_{P'} * \delta_{P^+}=h_P
    \end{align*}
\end{proof}

\begin{proposition}\label{mainthm0}
    For each $h\in A^0 (G)$, we have $[A_h]\in \mathcal{Z}^0(G) \subset \mc Z(G) \simeq  \emph{End}_{\HH(G)^{2}}(\HH(G))$, and the assignment $h\mapsto [A_h]$ defines an algebra map 
        \begin{align*}
            [A^0]\: :\: A^0(G) &\longrightarrow \mc Z^0(G) \\
            h &\longmapsto [A_h]
        \end{align*}
\end{proposition}
    \begin{proof}
     We already have that $[A_h]$ is well-defined from the first part, and it can be easily observed from the definition of $[A_h]$ that $[A_h](f*g)=[A_h](f)*g$. 
     \begin{claim}\label{gmod}
         Given $g\in G(\mathbbm k)$ and arbitrary $f\in \HH(G),\: h\in A^0(G)$, we have 
         \[
         Ad_g([A_h](f))=[A_h](Ad_g(f)) 
         \]
     \end{claim}
      We first show that for a fixed $P_0 \in Par $ and arbitrary $f\in \HH(G),\: h\in A^0(G)$, we have 
         \[
         Ad_p([A_h](f))=[A_h](Ad_p(f)) \: \text{ for }p\in P_0.
         \]
     Choose $Y\in \Omega $ large enough such that $[A_h](f)=[A^Y_h]*f$ and $[A_h](Ad_p(f))=[A^Y_h]*(Ad_p(f))$. Since $Y$ is $I$-invariant and $|P_0/I|$ is finite, $\exists \: Y^{P_0}\supset Y$ such that $Y^{P_0}$ is $P_0$-invariant. Then, $[A_h](f)= [A^{Y^{P_0}}_h]*f$ and $Ad_p([A^{Y^{P_0}}_h])=[A^{Y^{P_0}}_h]$ for $p\in P_0$ since $Y^{P_0}$ is $P_0$-invariant. So,
     \begin{align*}
         Ad_p([A_h](f))&=Ad_p([A^{Y^{P_0}}_h]*f)\\
         &=Ad_p([A^{Y^{P_0}}_h])*Ad_p(f)\\
         &=[A^{Y^{P_0}}_h]*Ad_p(f)\\
         &=[A_h](Ad_p(f)) \: \text{ since } Y^{P_0}\supset Y
     \end{align*}
     From the above, we see that $[A_h]$ is a $P$-module map (where $P$ acts by conjugation) for each $P\in Par$. Since $G(\mathbbm k)$ is generated by $\{P\:|\: P\in Par\}$, we can write $g\in G(\mathbbm k)$ as $g=\prod_{i=1}^k p_i$ where $p_i\in P_i$ for $P_i\in Par$. Then,
     \begin{align*}
       Ad_g([A_h](f))&=  Ad_{p_1}\circ \cdots \circ Ad_{p_k}\left([A_h](f)\right)\\
       &= [A_h]\left(Ad_{p_1}\circ \cdots \circ Ad_{p_k}(f)\right)\\
       &= [A_h](Ad_g(f))
     \end{align*}
which shows that $[A_h]$ is a $G(\mathbbm k)$-module map (where $G(\mathbbm k)$ acts by conjugation) and finishes the proof of the claim. \par 
Let $\mathcal R$ denote the action of $G(\mathbbm k)$ on $\HH(G)$ by right translation, i.e., $\mc R_g(f)(x)= f(xg)$ for $g, \: x\in G(\mathbbm k)$ and $f\in \HH(G)$. We can see from the definition of $[A_h]$ that $\mc R_g([A_h](f))=[A_h](\mc R_g(f))$. So, for the $G(\mathbbm k)^2$ action on $\HH(G)$ defined by $((g,h)f)(x)=f(g^{-1}xh)$, we have that $[A_h]$ is a $G(\mathbbm k)^2$ module map $\HH(G) \longrightarrow \HH(G)$, and hence a $\HH(G)^2$-module map for the $\HH(G)^2$-module structure on $\HH(G)$ induced from the same $G(\mathbbm k)^2$ action on $\HH(G)$. This shows that $[A_h]\in  \text{End}_{\HH(G)^{2}}(\HH(G)) \simeq \mc Z(G)$.\par
The next step is showing that the map $[A^0] : A^0(G) \longrightarrow \mc Z(G)$ is an algebra map, and then we will finally show that the image lies in the depth-zero part. Note that there is a natural convolution product defined on $A^0(G)$ which gives $A^0(G)$ its algebra structure. For $h=\{h_P\}$ and $\:h'=\{{h'}_P\}\in A^0(G)$, $h*h'=\{h_P* {h'}_P\}$. 
\begin{claim}\label{alg}
 Given $h,\: h'\in A^0(G)$, we have $[A_{h*h'}]=[A_h]\circ [A_{h'}]$   
\end{claim}
We will show $[A_{h*h'}]=[A_h]\circ [A_{h'}]$ via a series of reductions. First observe that in order to prove the claim, it is enough to show that for $Y\in \Omega $
    \begin{equation}\label{alg1}
        [A_h](Av^{Y_P}({h'}_P))= Av^{Y_P}(h_P * {h'}_P) \:\: \forall P\in Par
    \end{equation}
This is because given $f \in \HH (G)$, we can choose $Y \in \Omega $ large enough such that $[A_{h'}](f)=[A^Y_{h'}]*f$ and $[A_{h*h'}](f)=[A^Y_{h*h'}]*f $. Then, if \eqref{alg1} is true, we have 
\begin{align*}
    [A_{h*h'}](f)&=[A^Y_{h*h'}]*f\\
    &= \sum_{P\in Par}(-1)^{r(G)-r(P)}\left( Av^{Y_P}(h_P*h_{P'})* f\right)\\
    &= \sum_{P\in Par}(-1)^{r(G)-r(P)} \left( [A_h](Av^{Y_P}({h'}_P)) *f \right)\\
    &= \sum_{P\in Par}(-1)^{r(G)-r(P)} [A_h]\left( Av^{Y_P}({h'}_P) *f \right)\\
    &= [A_h]\left(\sum_{P\in Par}(-1)^{r(G)-r(P)}  Av^{Y_P}({h'}_P) *f \right)\\
    &= [A_h]\left([A^Y_{h'}]*f\right)\\
    &= [A_h]\left([A_{h'}](f)\right)   
\end{align*}
So, we have reduced the proof of the claim to the proof of \eqref{alg1}. Further, observe that in order to show that \eqref{alg1} is true, it is enough to show that 
\begin{equation}\label{alg2}
    Ad_y\circ [A_h] = [A_h]\circ Ad_y
\end{equation}
for $y\in G(\mathbbm k)$. This is because if \eqref{alg2}, is true, we have 
\begin{align*}
    [A_h](Av^{Y_P}({h'}_P))&= Av^{Y_P}([A_h]({h'}_P))\\
    &= Av^{Y_P}([A_h](\delta_{P^+} * {h'}_P))\\
    &= Av^{Y_P}([A_h](\delta_{P^+}) * {h'}_P)\\
    &= Av^{Y_P}(h_P * {h'}_P)
\end{align*}
which shows that \eqref{alg1} is true, and we are done. However, we know that \eqref{alg2} is true from Claim \ref{gmod} which finishes the proof of the claim. \par
We are only left to show that the image of $[A^0]\: :\: A^0(G) \rightarrow \mc Z(G)$ lies in the depth-zero part, i.e., $[A_h] \in \mc Z^0(G)$. Let $\delta := \{ \delta_{P^+}\}_{P\in Par} \in A^0(G)$. As per Theorem 4.4.1 in \cite{BKV}, we know that the element $[A_{\delta}]$ is the projector to the depth-zero part of the Bernstein center. Since $h*\delta =h$ for $h \in A^0(G)$, we have $[A_h]= [A_{h*\delta}]= [A_h]\circ [A_\delta]$, and hence $[A_h] \in \mc Z^0(G)$.
\end{proof}

\begin{theorem}\label{main thm depth 0}
    The map $[A^0]: A^0(G) \rightarrow \ZZ^0(G) $ defined in Proposition \ref{mainthm0} is an algebra isomorphism onto the depth-zero Bernstein center.
\end{theorem}

The  theorem follows directly from the following propositions \ref{sect} and \ref{inj}. 

\begin{proposition}\label{sect}
    We have an algebra map 
    \begin{align*}
       \Psi^0 \::\: \mc Z^0(G) \longrightarrow A^0(G)
    \end{align*}
    such that $\Psi^0 \circ [A^0] = Id_{A^0(G)}$
\end{proposition}
\begin{proof}
Given $z \in \mc Z(G)$, let $z_\HH$ be the image of $z$ under the following algebra isomorphism
\begin{align*}
    \ZZ(G) &\xlongrightarrow{\sim} \text{End}_{\HH(G)^{2}}(\HH(G))\\
    z & \longmapsto z_\HH
\end{align*}
    For $P\in Par $, define 
    \begin{align*}
        \Psi^0_P \::\: \mc Z^0(G) &\longrightarrow \M^0_P\\
        z &\longmapsto z_\HH(\delta_{P^+})
    \end{align*}
    We can easily see that the map $\Psi^0_P$ is well-defined since $P^+$ is normal in $P$. Let $\pi^0_P:A^0(G) \rightarrow \M^0_P $ be the canonical projection map. For $P,\:Q \in Par $ such that $P\subset Q $, we have $\Psi^0_P= \phi^0_{P,Q} \circ \Psi^0_Q $. So, there exists a map $\Psi^0 := \lim_{P\in Par}\Psi^0_P$ 
    \[
    \Psi^0 : \ZZ^0(G) \longrightarrow A^0(G)
    \]
    such that $\pi^0_P \circ \Psi^0 = \Psi^0_P$ for $P \in Par$. 
    Further, note that each $\Psi^0_P$ is an algebra map. Given $z,z' \in \ZZ^0(G)$,
    \begin{align*}
        (z\circ z')_\HH(\delta_{P^+}) &= z_\HH \circ {z'}_\HH(\delta_{P^+})=z_\HH({z'}_\HH(\delta_{P^+}*\delta_{P^+}))\\
        &= z_\HH(\delta_{P^+}*{z'}_\HH(\delta_{P^+})) =z_\HH(\delta_{P^+})*{z'}_\HH(\delta_{P^+})
    \end{align*}
    Hence, $\Psi^0 : \ZZ^0(G) \rightarrow A^0(G)$ is an algebra map. Finally, note that for $h=\{h_P\}_{P\in Par} \in A^0(G)$, we have from Proposition \ref{eval} that  $\Psi^0_P([A_h])= [A_h](\delta_{P^+})= h_P$ $\forall \: P \in Par$. Hence, 
    \[
    \Psi^0 \circ [A^0] (h) = \Psi^0([A_h])=h
    \]
    for $h=\{h_P\}_{P\in Par} \in A^0(G)$, and we are done.
\end{proof}
The above proposition implies that $\Psi^0$ is surjective and $[A^0]$ is injective as algebra maps. Observe that if we can show injectivity of $\Psi^0$, we can conclude that $\Psi^0$ and $[A^0]$ are inverse algebra isomorphisms.
\begin{proposition}\label{inj}
    $\Psi^0$ defined in the previous proposition is injective. 
\end{proposition}
\begin{proof}
    Assume $\Psi^0 (z) = \Psi^0 (z')$ for $z,\:z' \in \ZZ^0(G)$ and let $(\pi,\; V)$ be a smooth irreducible $G(\mathbbm k)$ representation of depth zero. $\exists \:P \in Par $ such that $V^{P^+}\ni v \neq \{0\}$. Then $\delta_{P^+}(v)=v$. In order to show $z=z'$, it is enough to show $z_V(v) = {z'}_V(v)$, since by Schur's lemma $z|_V = f_z(\pi)Id_V$ for some $f_z \in Fun(\:Irr (G) , \; \C)$. Note that $z_V(v)=z_V(\delta_{P^+}(v))=z_\HH(\delta_{P^+})(v)$ and the same is true for $z'$. Since $\Psi^0 (z) = \Psi^0 (z')$, we have 
    \[
    \pi^0_P \circ \Psi^0 (z) = \pi^0_P \circ \Psi^0(z') \Rightarrow \Psi^0_P(z)=\Psi^0_P(z')\Rightarrow z_\HH(\delta_{P^+})= {z'}_\HH(\delta_{P^+})
    \]
    $\forall \; P \in Par$. Hence,
    \[
    z_V(v) = z_\HH(\delta_{P^+})(v)= {z'}_\HH(\delta_{P^+})(v)= {z'}_V(v)
    \]
    which proves injectivity of $\Psi^0$. 
\end{proof}

\section{A description of the positive integral-depth
Bernstein center}\label{description depthr}

\subsection{Stabilization in the positive integral depth case}
Our setting remains the same as in the previous section. In the subsequent parts of this section, we fix $r$ to be a positive integer. Define 
$$\mathcal{M}^r_{P} := C^{\infty}_c(\frac{G(\mathbbm k)/P_r^+}{P}) $$
to be the algebra of compactly supported smooth functions on $G(\mathbbm k)$ which are $P_r^+$ bi-invariant and $P$ conjugation invariant.
For any $P\subset Q \in Par$ , we have a map 
\begin{align*}
\phi^r_{P,Q}:\M^r_{Q} &\longrightarrow \M^r_{P}\\
h &\longmapsto h*\delta_{P_r^+}  
\end{align*}
With the above defined maps, we have an inverse system $\{\M^r_{P}\}_{P \in Par}$ and we define $A^r(G)$ to be the inverse limit of the algebras $\M^r_{P}$.
$$ A^r(G) := \lim_{P\in Par} \M^r_{P}$$
For $h=\{h_P\}_{P\in Par}\in A^r(G)$ and $Y\in \Omega$, we have the exact same definitions for $Av^{Y_P}(h_P)$ and $[A^Y_h]$.

We have the following generalization of  Theorem \ref{thmstab0} to positive depth.

\begin{theorem}\label{thm1r}
    For every $f\in \mc H(G)$ and $h\in A^r(G)$, the sequence $\{[A^Y_h] * f\}_{Y \in \Omega} $ stabilizes, and hence $\lim_{Y\in \Omega} \:[A^Y_h] * f $ is well-defined.
\end{theorem}

\begin{lemma}\label{lemdelr}
    Let $w\in \Tilde{W}, \: \alpha \in \Tilde{\Delta}, \: Q\in \:Par $ and $n\in \mathbb{N}$. Let $J \subset J_w \backslash  \alpha $ be such that $U_{w(\alpha)} \subset Q^+_{n}, \: J \neq \Tilde{\Delta} \backslash \alpha $ and $J' = J \cup \{\alpha \}$. Then 
    \begin{equation}
        Av^{(Y_w)_{P_{J'}}}(\delta_{({P_{J'})_r^+}}) \:*\: \delta_{Q^+_{n+r}} = Av^{(Y_w)_{P_{J}}}(\delta_{({P_{J})_r^+}}) \:* \:\delta_{Q^+_{n+r}}
    \end{equation}
\end{lemma}
\begin{proof}
    Let $x, y, z \in \X$ be such that $P_{J} = G_x, P_{J'}=G_y $ and $w^{-1}Qw = G_z$. 
\begin{claim}
        $(P_J)_r^+ = (P_{J'})_r^+ \:( (P_J)_r^+ \cap w^{-1}Q^+_{n+r} w)$
    \end{claim}
     It suffices to show that for every $\beta \in \Tilde{\Phi}$ and $U_\beta \subset (P_J)_r^+\backslash (P_{J'})_r^+$, we have $U_\beta \subset w^{-1}Q^+_{n+r} w$. This is equivalent to showing that for $\beta \in \Tilde{\Phi}$ such that $\beta(x) >r$, either $\beta(y) > r $ or $\beta(z) >n+r$. Note that what we essentially proved in the analogous claim \ref{cl1lemdel} in the depth-zero case is that if $\beta \in \Tilde{\Phi}$ such that $\beta(x)> 0$, then either $\beta(y)>0$ or $\beta(z) >n $. Now let $\beta \in \Tilde{\Phi}$ such that $ \beta(x)>r$. Since $r\in \Z$, $\beta'=\beta -r \in \Tilde{\Phi}$ and $\beta'(x)>0$. So, from claim \ref{cl1lemdel}, we see that either $\beta'(y) >0 $ or $\beta'(z) >n$,which implies that either $\beta(y)>r$ or $\beta(z)>n+r$ and proves our claim.\par
     The above claim shows that $(P_{J'})_r^+\cdot w^{-1}Q_{n+r}^+ w= (P_{J})_r^+ \cdot w^{-1}Q_{n+r}^+ w$, and using ideas in claim \ref{cl3lemdel},  we see that 
     \begin{equation}\label{eqdelr}
         \delta_{(P_{J'})_r^+}\:*\:\delta_{w^{-1}Q_{n+r}^+ w} = \delta_{(P_{J})_r^+}\:*\:\delta_{w^{-1}Q_{n+r}^+ w}.
     \end{equation}
     Following the arguments in the proof of the Lemma \ref{lemdel}, we see that we can again fix a set $I_w\subset I $ such that $I_w w$ forms a set of representatives of $Y_w$, and also for $(Y_w)_{P_J}$ for $J\subset J_w$. Hence,
     \begin{equation}
     Av^{I_w w}(h_{P_J})= Av^{(Y_w)_{P_J}}(h_{P_J}) \text{ for } h_{P_J} \in \M^r_{P_J}, \: J\subset J_w
\end{equation}
Since $Q$ normalises $Q_{n+r}^+$, the same steps as in claim \ref{cl2lemdel} gives us 
  \[
  Av^{(Y_w)_{P_{J}}}(\delta_{(P_{J})_r^+}) \:* \delta_{Q_{n+r}^+}= Av^{I_w w}(\delta_{(P_{J})_r^+}\:*\:\delta_{w^{-1}Q_{n+r}^+ w} )
  \] 
  and the same is true for $J'$ since $J,J' \subset J_w$. Thus, \eqref{eqdelr} finishes the proof of the lemma. 
\end{proof}

\begin{lemma}\label{lemhr}
    Let $w\in \Tilde{W}, \: \alpha \in \Tilde{\Delta}, \: Q\in \:Par $ and $n\in \mathbb{N}$. Let $J \subset J_w \backslash  \alpha $ be such that $U_{w(\alpha)} \subset Q^+_{n}, \: J \neq \Tilde{\Delta} \backslash \alpha $ and $J' = J \cup \{\alpha \}$. Let $h=\{ h_P\}_{P\in Par} \in A^r (G)$. Then 
    \begin{equation}\label{eqlemhr}
        Av^{(Y_w)_{P_{J'}}}(h_{P_{J'}}) \:*\: \delta_{Q^+_{n+r}} = Av^{(Y_w)_{P_{J}}}(h_{P_{J}}) \:* \:\delta_{Q^+_{n+r}}
    \end{equation}
\end{lemma}
\begin{proof}
   Using the same idea as in lemma \ref{lemdelr}, we see that 
   \[
    Av^{(Y_w)_{P_{J}}}(h_{P_{J}}) \:*\: \delta_{Q^+_{n+r}} = Av^{I_w w}(h_{P_{J}}\:*\:\delta_{w^{-1}Q^+_{n+r} w} )
   \]
   and the same is true for $J'$. Thus following lemma \ref{lemh}, it suffices to show that 
   \[
   h_{P_{J}}\:*\:\delta_{w^{-1}Q^+_{n+r} w} = h_{P_{J'}}\:*\:\delta_{w^{-1}Q^+_{n+r} w}
   \]
   From \eqref{eqdelr} in lemma \ref{lemdelr}, we have $ \delta_{(P_{J'})_r^+}\:*\:\delta_{w^{-1}Q_{n+r}^+ w} = \delta_{(P_{J})_r^+}\:*\:\delta_{w^{-1}Q_{n+r}^+ w}$. Using that and the fact that $h_{P_{J'}}*\delta_{(P_{J})_r^+}=h_{P_{J}}$, we get 
   \begin{align*}
    h_{P_{J'}}\:*\:\delta_{w^{-1}Q_{n+r}^+ w} &= h_{P_{J'}}\:*\:(\delta_{(P_{J'})_r^+}\: * \: \delta_{w^{-1}Q_{n+r}^+ w})\\
    &= (h_{P_{J'}}\:*\:\delta_{(P_{J})_r^+})\: * \: \delta_{w^{-1}Q_{n+r}^+ w}\\
    &= h_{P_{J}}\:*\:\delta_{w^{-1}Q_n^+ w}
\end{align*}
and we are done.  
\end{proof}
\begin{proposition}\label{stabr}
    Let $Q\in Par, \: Y\in \Omega $ and $h=\{ h_P\}_{P\in Par}\in A^r (G)$. If $Y\supset Y(Q_n^+)$, we have 
    \begin{equation}
        [A^Y_h] \:*\: \delta_{Q_{n+r}^+} = [A^{Y(Q_n^+)}_h]\:*\: \delta_{Q_{n+r}^+}
    \end{equation}
\end{proposition}
\begin{proof}
    We use the same idea as in the proof of proposition \ref{stab}, and proceed by induction on the number of $I-$orbits in $Y\backslash Y(Q_n^+)$. For $Y\in \Omega$ and $w\in \Tilde{W} \backslash S(Q_n^+)$ chosen such that $Y_w\subset Y$ is maximal, $Y'=Y\backslash Y_w \in \Omega $ and it is enough to show that 
    \[
    [A^Y_h] \:*\: \delta_{Q_{n+r}^+} = [A^{Y'}_h]\:*\: \delta_{Q_{n+r}^+}
    \]
    The exact same steps as in the first part of claim \ref{cl1stab} gives us\[
     [A^Y_h]= [A^{Y'}_h] + [A^w_h],\]
     where $[A^w_h]= \sum_{J\subset J_w} (-1)^{r(G)-|J|} Av^{(Y_w)_{P_J}}(h_{P_J})$. Hence, as in the depth-zero case, it is enough to show that $[A^w_h]*\delta_{Q_{n+r}^+}=0\:\forall w \in \Tilde{W} \backslash S(Q_n^+)$ such that $Y_w \subset Y$. By definition of $S(Q_n^+)$, for each $w \in \Tilde{W} \backslash S(Q_n^+)$, $\exists \alpha \in \Tilde{\Delta}$ such that $U_{w(\alpha)} \subset Q_n^+$, and hence $\alpha \in J_w$. Let $J'=J\cup \alpha$. Using Lemma \ref{lemhr} and following the steps in the proof of claim \ref{cl1stab} in lemma \ref{stab}, we see
\begin{align*}
[A^w_h] * \delta_{Q_{n+r}^+} &= \sum_{J\subset J_w} (-1)^{r(G)-|J|} \left(Av^{(Y_w)_{P_J}}(h_{P_J}) \: * \: \delta_{Q_{n+r}^+} \right) \\
&= \sum_{J\subset J_w\backslash \alpha } (-1)^{r(G)-|J|} \left(\left(Av^{(Y_w)_{P_J}}(h_{P_J}) *\delta_{Q_{n+r}^+}\right)- \left(Av^{(Y_w)_{P_{J'}}}(h_{P_{J'}})* \delta_{Q_{n+r}^+}\right) \right)\\
&=0,
\end{align*}
which finishes the proof of the proposition.
\end{proof}
Using the above lemmas and propositions, we can complete the proof of Theorem \ref{thm1r}.
\begin{proof}[Proof of Theorem \ref{thm1r}]
     We are trying to prove that for every $f\in \mc H(G)$ and $h\in A^r(G)$, the sequence $\{[A^Y_h] * f\}_{Y \in \Omega} $ stabilizes. For any $f\in \mc H(G)$, $\exists \: n\in \N $ such that $f$ is left $I_{n+r}^+$-invariant, i.e., $f= \delta_{I_{n+r}^+} * f$. So,
    \[
    [A^Y_h] * f = [A^Y_h] * \delta_{I_{n+r}^+} * f
    \]
    Now, from Proposition \ref{stabr}, we observe that 
    \[
    [A^Y_h] * \delta_{I_{n+r}^+} = [A^{Y(I_n^+)}_h]* \delta_{I_{n+r}^+}
    \]
    for for $n \in \N $ and large enough $Y \in \Omega $ such that $Y \supset Y(I_n^+)$, since $Y(I_n^+)$ is finite by Lemma \ref{lemfin}. Hence, for $Y \supset Y(I_n^+)$, the same steps as in the depth-zero case gives us  
    \begin{align*}
        [A^Y_h] * f &= [A^Y_h] * \delta_{I_{n+r}^+} * f\\
        &=[A^{Y(I_n^+)}_h]\:*\: \delta_{I_{n+r}^+} *f\\
        &=[A^{Y(I_n^+)}_h] *f
    \end{align*}
     and $\{[A^Y_h] * f\}_{Y \in \Omega} $ stabilizes. 
\end{proof}

\subsection{A limit description of the positive depth Bernstein center}
As a consequence of Theorem \ref{thm1r}, we see that $\lim_{Y\in \Omega} \:[A^Y_h] * f $ is well-defined. So, for each $h\in A^r (G)$, we can define $[A_h] \in \text{End}_{\HH(G)^{op}}(\HH(G))$ by the formula         
\begin{equation}
    [A_h](f):=\lim_{Y\in \Omega} \:[A^Y_h] * f.
\end{equation}
\begin{proposition}\label{evalr}
    Given $h =\{h_P\}_{P\in Par}\in A^r(G),\: Y\in \Omega $ and $P\in Par$, we have 
    \begin{equation}
        [A^Y_h] * \delta_{P_r^+} = h_P
    \end{equation}
\end{proposition}
\begin{proof}
    We have from proposition \ref{stabr} that $\forall\: Y\in \Omega$
    \[
    [A^Y_h] * \delta_{P_r^+} = [A^{Y_1}_h] * \delta_{P_r^+}
    \]
    since $Y(P^+) = Y_1$. Let $\alpha \in \Tilde{\Delta}$ be such that $U_{\alpha} \subset P^+$, $J'=J\cup \{\alpha\}$  and denote $P_{\{\Tilde{\Delta}\backslash \alpha\}} \in Par$ by $P'$. Then
    \begin{align*}
        [A^{Y_1}_h] &= \sum_{P \in Par} (-1)^{r(G)-r(P)}Av^{{(Y_1})_P}(h_P)\\
        &= Av^{(Y_1)_{P'}}(h_{P'}) +\sum_{J\subsetneq \Tilde{\Delta}\backslash \alpha} (-1)^{r(G)-|J|} \left(Av^{(Y_1)_{P_J}}(h_{P_J}) - Av^{(Y_1)_{P_{J'}}}(h_{P_{J'}}) \right)\\
        &= h_{P'} + \sum_{J\subsetneq \Tilde{\Delta}\backslash \alpha} (-1)^{r(G)-|J|} (h_{P_J}-h_{P_{J'}})
    \end{align*}
    Using the same arguments as in the proof of proposition \ref{eval}, we see that $P \subset P'$ and hence $h_{P'} * \delta_{P_r^+}=h_P$. So, if we can prove that $h_{P_J} * \delta_{P_r^+} = h_{P_{J'}} * \delta_{P_r^+}$, we are done. Note that $h_{P_J} = h_{P_{J'}} * \delta_{(P_J)_r^+}$ since $P_{J} \subset P_{J'}$, which gives us
\begin{align*}
    h_{P_J} * \delta_{P_r^+} - h_{P_{J'}} *\delta_{P_r^+} &= h_{P_{J'}} * \delta_{(P_J)_r^+}* \delta_{P_r^+}-h_{P_{J'}} *\delta_{(P_{J'})_r^+} *\delta_{P_r^+}\\
    &=h_{P_{J'}} * \left(\delta_{(P_J)_r^+}* \delta_{P_r^+}- \delta_{(P_{J'})_r^+} *\delta_{P_r^+}\right)
\end{align*}
meaning that it is enough to show $\delta_{(P_J)_r^+}* \delta_{P_r^+}= \delta_{(P_{J'})_r^+} *\delta_{P_r^+}$ to complete the proof of the proposition.
\begin{claim}\label{cl1evalr}
    $\delta_{(P_J)_r^+}* \delta_{P_r^+}= \delta_{(P_{J'})_r^+} *\delta_{P_r^+}$
\end{claim}
Let $x,\: y,\: z \in \X$ be such that $P_J =G_x,\: P_{J'} = G_y $ and $P=G_z$. To prove the claim it is enough to show that $(P_J)_r^+ \cdot P_r^+ =(P_{J'})_r^+ \cdot P_r^+ $. We already know that $(P_J)_r^+ \cdot P_r^+ \supset (P_{J'})_r^+ \cdot P_r^+ $. To show the reverse inclusion,let $\beta \in \tilde{\Phi} $ be such that $U_\beta \subset (P_J)_r^+ \backslash (P_{J'})_r^+ $, and we show that $U_\beta \subset P_r^+ $. This is equivalent to showing that $\beta(x) >r $ implies that either $\beta(y)>r$ or $\beta(z)>r$. Note that in claim \ref{cl1eval}, we proved that if $\beta(x) >0$, then $\beta(y) >0$ or $\beta(z)>0$. Now, if $\beta(x)>r$, consider $\beta'=\beta-r \in \Tilde{\Phi}$. Then, $\beta'(x)>0$, which implies $\beta'(y)>0$ or $\beta'(z)>0$, and hence either $\beta(y)>r$ or $\beta(z)>r$ finishing the proof of the claim. 
\end{proof}
\begin{remark}\label{rmkevalr}
Note that proposition \ref{evalr} implies that for $h =\{h_P\}_{P\in Par}\in A^r(G)$ and $\forall P\in Par$, we have $[A_h](\delta_{P_r^+})= h_P$. 
\end{remark}
\begin{proposition}\label{thm2r}
    For each $h\in A^r(G)$, we have $[A_h]\in \emph{End}_{\HH(G)^{2}}(\HH(G))  \simeq \mc Z(G)$, and the assignment $h\mapsto [A_h]$ defines an algebra map 
        \begin{align*}
            [A^r]\: :\: A^r(G) &\longrightarrow \mc Z(G) \\
            h &\longmapsto [A_h]
        \end{align*} 
\end{proposition}
\begin{proof}
   It can be easily observed from the definition of $[A_h]$ that $[A_h](f*g)=[A_h](f)*g$ for $h\in A^r(G)$. Since, $[A_h]$ commutes with right convolutions, it is enough to show that it is $G(\mathbbm k)$-conjugation equivariant.  
     \begin{claim}\label{gmodr}
         Given $g\in G(\mathbbm k)$ and arbitrary $f\in \HH(G),\: h\in A^r(G)$, we have 
         \[
         Ad_g([A_h](f))=[A_h](Ad_g(f)) 
         \]
     \end{claim}
     This claim and the proof of it is exactly the same as the claim \ref{gmod} in the depth-zero case. Using the same ideas, we have that for $h\in A^r(G)$, $[A_h]\in  \text{End}_{\HH(G)^{2}}(\HH(G)) \simeq \mc Z(G)$.\par
    The next step is showing that the map $[A^r] : A^r(G) \longrightarrow \mc Z(G)$ is an algebra map. 
     \begin{claim}\label{algr}
 Given $h,\: h'\in A^r(G)$, we have $[A_{h*h'}]=[A_h]\circ [A_{h'}]$   
\end{claim}
Again, following the steps in the proof of claim \ref{alg}, we see that it is enough to show 
\begin{equation}\label{alg1r}
        [A_h](Av^{Y_P}({h'}_P))= Av^{Y_P}(h_P * {h'}_P) \:\: \forall P\in Par,
    \end{equation}
    where $h,\: h'\in A^r(G)$. Further, observe that similar to the depth-zero case, it is enough to show that 
    \begin{equation}\label{alg2r}
    Ad_y\circ [A_h] = [A_h]\circ Ad_y
    \end{equation}
    for $y\in G(\mathbbm k)$. This is because,if \eqref{alg2r} is true, we have using proposition \ref{evalr} that 
    \begin{align*}
    [A_h](Av^{Y_P}({h'}_P))&= Av^{Y_P}([A_h]({h'}_P))\\
    &= Av^{Y_P}([A_h](\delta_{P_r^+} * {h'}_P))\\
    &= Av^{Y_P}([A_h](\delta_{P_r^+}) * {h'}_P)\\
    &= Av^{Y_P}(h_P * {h'}_P)
\end{align*}
    Now, \eqref{alg2r} is certainly true from claim \ref{gmodr}, which finishes the proof of the proposition. 
\end{proof}
\begin{proposition}\label{stabrv}
    For every $V \in R(G)$ and $v\in V$, $\{[A^Y_h] (v)\}_{Y \in \Omega} $ stabilizes and $[A_h](v)= \lim_{Y \in \Omega}[A^Y_h] (v)$.
\end{proposition}
\begin{proof}
    Choose $f\in \HH(G)$ such that $fv=v$. Then, $$\{[A^Y_h] (v)\}_{Y \in \Omega} = \{[A^Y_h] (fv)\}_{Y \in \Omega} = \{([A^Y_h]*f) (v)\}_{Y \in \Omega}. $$ Since,  $\{([A^Y_h]*f)\}_{Y \in \Omega}$ stabilizes by Theorem \ref{thm1r}, $\{[A^Y_h] (v)\}_{Y \in \Omega}$ also stabilizes and $\lim_{Y \in \Omega}[A^Y_h] (v)$ is well-defined. Moreover, we have 
    $$[A_h](v)=[A_h](fv)=([A_h](f))(v)=\lim_{Y \in \Omega}([A^Y_h]*f)(v)= \lim_{Y \in \Omega}([A^Y_h](fv))= \lim_{Y \in \Omega}[A^Y_h] (v)$$ which gives us the equality. 
\end{proof}
\begin{proposition}
    Let $\delta_r$ denote $\{\delta_{P_r^+}\}_{P\in Par} \in A^r(G)$. Then, $[A_{\delta_r}]$ is the projector to the depth-$r$ part of the Bernstein center, i.e., $[A_{\delta_r}]= z^r \in \ZZ^r(G)\subset \ZZ(G)$
\end{proposition}
\begin{proof}We will prove the proposition in 2 steps. 
    \begin{claim}
        For every $(\pi,V) \in Irr(G)_{\leq r}$, $[A_{\delta_r}]|_V= Id_V$
    \end{claim}
    Let $(\pi,V) \in Irr(G)_{\leq r}$. Then, $\exists \; P \in Par$ and $0\neq v \in V^{P_r^+}$ and to prove the claim, it is enough to show that $[A_{\delta_r}](v)=v$. Since $v \in V^{P_r^+}$, $\delta_{P_r^+}(v)=v$ and using remark \ref{rmkevalr}, we have
    \[
    [A_{\delta_r}](v)= [A_{\delta_r}](\delta_{P_r^+}(v))= \left([A_{\delta_r}](\delta_{P_r^+})\right)(v)= \delta_{P_r^+}(v)=v
    \]
    which finishes the proof of the claim. 
    \begin{claim}
        For every $(\pi,V) \in Irr(G)_{> r}$, $[A_{\delta_r}]|_V= 0$
    \end{claim}
    Let $v\in V$. Since $(\pi,V) \in Irr(G)_{> r}$, $V^{(G_x)_r^+}=0$ for any $x\in \X$. and hence $\delta_{(G_x)_r^+}(v)=0 \:\forall \: v\in V$ and $\forall \: x \in \X$. We will show that $[A_{\delta_r}](v)=0$. We know from proposition \ref{stabrv} that $\{[A^Y_h] (v)\}_{Y \in \Omega} $ stabilizes, and so if we choose $Y \in \Omega$ large enough, we have $[A_{\delta_r}](v)= [A^Y_{\delta_r}](v)$. Since $Ad_y(\delta_{P_r^+}) = \delta_{yP_r^+y^{-1}}=\delta_{(G_x)_r^+}$ for some $x\in \X$, we see from the structure of $[A^Y_{\delta_r}]$ that it is a finite sum of elements of the form $\delta_{(G_x)_r^+}$ with some signs in front of the terms, i.e., 
    \[
    [A^Y_{\delta_r}]= \sum_x (-1)^{r(G)-r(G_x)}\delta_{(G_x)_r^+}
    \]
    where $x$ runs over some finite set in $\X$. Hence,
    \[
    [A_{\delta_r}](v)= [A^Y_{\delta_r}](v)= \sum (-1)^{r(G)-r(G_x)}\delta_{(G_x)_r^+}(v)=0
    \]
    which proves the claim and completes the proof of the proposition. 
\end{proof}

\begin{remark}
    The projector that we define and construct here is same as the one constructed in \cite{BKV2} if $r$ is taken to be an integer. Note that the construction in \cite{BKV2} also works for $r\in \Q_{\geq0}\backslash \Z_{\geq 0}$. Our construction uses the ideas in \cite{BKV} and extends them to positive integers. 
\end{remark}

\begin{theorem}\label{main thm depth r}
    The image of the map $[A^r]$ lies in the depth-$r$ part, i.e., for each $h\in A^r(G)$, we have $[A_h]\in \mc Z^r(G)$, and 
        \begin{align*}
            [A^r]\: :\: A^r(G) &\longrightarrow \mc Z^r(G) \\
            h &\longmapsto [A_h]
        \end{align*} 
        is an algebra isomorphism onto the depth-$r$ part of the Bernstein center. 
\end{theorem}
\begin{proof}
    Let $h\in A^r(G)$. Then, $$[A_h]=[A^r](h)= [A^r](h*\delta_r)= [A_{h*\delta_r}]= [A_h]\circ [A_{\delta_r}]$$
    and hence $[A_h]\in \mc Z^r(G)$ since $[A_{\delta_r}]$ is the depth-$r$ projector. To prove that the map is an isomorphism, we will follow the ideas in the depth-zero case and do it in 2 steps. 
    \begin{claim}
            We have an algebra map 
    \begin{align*}
       \Psi^r \::\: \mc Z^r(G) \longrightarrow A^r(G)
    \end{align*}
    such that $\Psi^r \circ [A^r] = Id_{A^r(G)}$
    \end{claim}
The proof of this claim follows the steps of proposition \ref{sect}. For $P\in Par $, define 
    \begin{align*}
        \Psi^r_P \::\: \mc Z^r(G) &\longrightarrow \M^r_P\\
        z &\longmapsto z_\HH(\delta_{P_r^+})
    \end{align*}
    We can easily see that the map $\Psi^r_P$ is well-defined since $P_r^+$ is normal in $P$. Let $\pi^r_P:A^r(G) \rightarrow \M^r_P $ be the canonical projection map. For $P,\:Q \in Par $ such that $P\subset Q $, we have $\Psi^r_P= \phi^r_{P,Q} \circ \Psi^r_Q $. So, there exists a map $\Psi^r := \lim_{P\in Par}\Psi^r_P$ 
    \[
    \Psi^r : \ZZ^r(G) \longrightarrow A^r(G)
    \]
    such that $\pi^r_P \circ \Psi^r = \Psi_P$ for $P \in Par$. 
    Further, note that each $\Psi^r_P$ is an algebra map. Given $z,z' \in \ZZ^r(G)$,
    \begin{align*}
        (z\circ z')_\HH(\delta_{P_r^+}) &= z_\HH \circ {z'}_\HH(\delta_{P_r^+})=z_\HH({z'}_\HH(\delta_{P_r^+}*\delta_{P_r^+}))\\
        &= z_\HH(\delta_{P_r^+}*{z'}_\HH(\delta_{P_r^+})) =z_\HH(\delta_{P_r^+})*{z'}_\HH(\delta_{P_r^+})
    \end{align*}
    Hence, $\Psi^r : \ZZ^r(G) \rightarrow A^r(G)$ is an algebra map. Finally, note that for $h=\{h_P\}_{P\in Par} \in A^r(G)$, we have from remark \ref{rmkevalr} that  $\Psi^r_P([A_h])= [A_h](\delta_{P_r^+})= h_P$ $\forall \: P \in Par$. Hence, 
    \[
    \Psi^r \circ [A^r] (h) = \Psi^r([A_h])=h
    \]
    for $h=\{h_P\}_{P\in Par} \in A^r(G)$, and we are done.\par
    The above claim implies that $\Psi^r$ is surjective and $[A^r]$ is injective as algebra maps. Observe that similar to the depth-zero case, if we can show injectivity of $\Psi^r$, we can conclude that $\Psi^r$ and $[A^r]$ are inverse algebra isomorphisms.
    \begin{claim}
        $\Psi^r$ is injective. 
    \end{claim}
    Assume $\Psi^r (z) = \Psi^r (z')$ for $z,\:z' \in \ZZ^r(G)$ and let $(\pi,\; V) \in Irr(G)_{\leq r}$ . $\exists \:P \in Par $ such that $V^{P_r^+}\ni v \neq \{0\}$. Then $\delta_{P_r^+}(v)=v$. In order to show $z=z'$, it is enough to show $z_V(v) = {z'}_V(v)$, since by Schur's lemma $z|_V = f_z(\pi)Id_V$ for some $f_z \in Fun(\:Irr (G), \; \C)$. Note that $z_V(v)=z_V(\delta_{P_r^+}(v))=z_\HH(\delta_{P_r^+})(v)$ and the same is true for $z'$. Since $\Psi^r (z) = \Psi^r (z')$, we have 
    \[
    \pi^r_P \circ \Psi^r (z) = \pi^r_P \circ \Psi^r(z') \Rightarrow \Psi^r_P(z)=\Psi^r_P(z')\Rightarrow z_\HH(\delta_{P_r^+})= {z'}_\HH(\delta_{P_r^+})
    \]
    $\forall \; P \in Par$. Hence,
    \[
    z_V(v) = z_\HH(\delta_{P_r^+})(v)= {z'}_\HH(\delta_{P_r^+})(v)= {z'}_V(v)
    \]
    which proves injectivity of $\Psi^r$ and gives us an isomorphism. 
\end{proof}
\begin{remark}
As a corollary of the proof, we 
            observe that for a given $h=\{h_P\}_{P\in Par}\in A^r(G)$, $h_P$ lies in the center $\ZZ(\HH_{P_r^+})$ of $\HH_{P_r^+}$, where $\HH_{P_r^+}= \delta_{P_r^+} * \HH(G)* \delta_{P_r^+}$. We can even directly see this from our construction since for a given $f \in \HH_{P_r^+}$, proving $f*h_P=h_P*f$ reduces to showing $A^Y_h*f=f*A^Y_h$ for some $Y \in \Omega$ large enough, and this is true if we choose $Y$ to be invariant under $\text{supp}(f)$ which we can since the support is compact. 
        \end{remark}

\subsection{A limit description of the Bernstein center}
We have an isomorphisms $[A^r] : A^r(G) \longrightarrow \mc Z^r(G)$ for all non-negative integers $r$. For any $r\in \Z_{\geq 0}$, we have a map 
\begin{align*}
            e_{r+1} \:: \:A^{r+1}(G) &\longrightarrow  A^r(G) \\
            \{h_P\}_{P\in Par} &\longmapsto \{h_P*\delta_{P_r^+}\}_{P\in Par}
        \end{align*} 
We define 
\begin{equation}
    A(G) := \lim_{r\in \Z_{\geq 0}}A^r(G)
\end{equation}
where the limit is taken with respect to the maps $e_{r}$. Further, note that we have a natural map 
\begin{align*}
            z_{r+1} \:: \:\ZZ^{r+1}(G) &\longrightarrow  \ZZ^r(G) \\
            z &\longmapsto z \circ [A_{\delta_r}]
        \end{align*} 
        such that $\ZZ(G) = \lim_{r\in \Z_{\geq 0}}\ZZ^r(G)$. 

        \begin{theorem}\label{main theorem}
            The algebra isomorphisms $[A^r] : A^r(G) \rightarrow \mc Z^r(G)$ fits into the following commutative diagram 
             \[
             \xymatrix{A^{r+1}(G) \ar[r]^{[A^{r+1}]}\ar[d]^{e_{r+1}}&\ZZ^{r+1}(G)\ar[d]^{z_{r+1}}\\
        A^r(G)\ar[r]^{[A^r]}&\ZZ^r(G)}
        \]
        In particular, we have an algebra isomorphism 
        \begin{equation}
            [A]=\lim_{r\in \Z_{\geq 0}}[A^r] : A(G) \longrightarrow \ZZ(G).
        \end{equation}
        \end{theorem}
        \begin{proof}
            Let $h = \{h_P\}_{P\in Par}\in A^{r+1}(G) $ and $h'=\{h_P*\delta_{P_r^+}\}_{P\in Par} = e_{r+1}(h)\in A^r(G)$. We are trying to show that $[A_{h'}]= [A_h]\circ [A_{\delta_r}]$. Let $(\pi, V)$ be a smooth irreducible representation of $G(\mathbbm k)$. If the depth of $\pi$ is $\leq r$, then  $\exists \: 0\neq v\in V^{P_r^+}$ for some $P\in Par$. In this case, \[
            [A_{h'}](v)=[A_{h'}](\delta_{P_r^+}(v))= [A_{h'}](\delta_{P_r^+})(v)= (h_P*\delta_{P_r^+})(v)=h_P(v)\]
            and \[
            [A_h]\circ [A_{\delta_r}](v)=[A_h](v)=[A_{h}](\delta_{P_{r+1}^+}(v))=[A_{h}](\delta_{P_{r+1}^+})(v)=h_P(v).
            \]
            So, we have $[A_{h'}]|_V=([A_h]\circ [A_{\delta_r}])|_V$. \par 
            In the case when the depth of $\pi $ is $>r$, we observe that both $[A_h]\circ [A_{\delta_r}]$ and $[A_{h'}]$ act by zero. The rest of the proof follows immediately. 
        \end{proof}
        \begin{remark}
            Consider the natural inclusion $\ZZ^r(G)\xhookrightarrow{i_r}\ZZ^{r+1}(G)$. Let $z \in \ZZ^r(G)$ be such that $z = [A^r]^{-1}(h)= [A_h]$ for $h=\{h_P\}_{P\in Par}\in A^r(G)$. Then $[A^{r+1}]^{-1}\circ i_r(z)=f \in A^{r+1}(G)$, where $f= \{f_P\}_{P\in Par}= \{A^{Y(P_1^+)}_h*\delta_{P_{r+1}^+}\}_{P\in Par}$. This is because $f_P= [A_h](\delta_{P_{r+1}^+})=A^{Y(P_1^+)}_h*\delta_{P_{r+1}^+} $ by Proposition \ref{stabr}. So the inclusion $j_r:=[A^{r+1}]^{-1}\circ i_r \circ [A^r]^{-1} : A^r(G) \rightarrow A^{r+1}(G)$ is given by 
            \begin{align*}
                 A^r(G) &\xlongrightarrow{j_r} A^{r+1}(G)\\
                \{h_P\}_{P\in Par}&\longmapsto \{A^{Y(P_1^+)}_h*\delta_{P_{r+1}^+}\}_{P\in Par}
            \end{align*}
            Note that $e_{r+1}\circ j_r =\text{Id}_{A^r(G)}$. 
        \end{remark}
        
\section{Stable functions}\label{stable functions}
In this section we introduce and study stable functions on finite reductive groups and finite reductive Lie algebras.

Let 
$k=\bar{\mathbb {F}}_q$. 
For the rest of the paper we fixed a square root $q^{-1/2}$ of $q$
and a non-trivial additive character $\psi:\mathbb {F}_q\to\mathbb C^\times$.
We also assume $p>>0$.

\subsection{Stable functions on finite reductive groups}
Let $\rG$ be a connected reductive group over $k$.
Let $ F:\rG\to\rG$ be the (geometric) Frobenious endomorphism associated to a $\mathbb F_q$-rational structure. Let $(\rT,\rB)$ be a 
$F$-stable Borel pair of $\rG$ and let $\rW=N(\rT)/\rT$
be the Weyl group. 
Let $par$ be the set of $F$-stable  parabolic subgroup $\rP\supset \rB$. 
For any $\rP\in par$ we denote by $\rU_P$ its unipotent radical and $\rL_\rP=\rP/\rU_P$ its Levi quiteint. The image of $\rT$ along 
$\rT\to\rB\to\rP\to\rL_\rP$ is a maximal torus
of $\rT_\rP\subset\rL_\rP$ and we write 
$\rW_\rP=N(\rT_\rP)/\rT_\rP$ the corresponding Weyl group.
We denote by ${\rG}^F, {\rB}^F$, etc the $F$-fixed points of $\rG$, $\rB$, etc.

Let $C(\rG^F)$ be the space of class functions on 
$\rG^F$ equipped with the  convolution product
\begin{equation}\label{convolution}
f\star f'(x)=\sum_{y\in\rG^F}f(xy^{-1})f'(y)   
\end{equation}
For any $\rG^F$-conjugation invariant subset 
$S\subset \rG^F$ we denote by $C(S)$ the set of $\rG^F$-conjugation invariant functions on $S$.
For any  $\rP\in par$, we have the parabolic restriction map
$\mathrm{res}^\rG_{\rL_\rP}:C(\rG^F)\to C(\rL_\rP^F)$
defined by  
\begin{equation}
\mathrm{res}^\rG_{\rL_\rP}(f)(\bar l)=\sum_{u\in \rU^F_\rP}f(lu)
\end{equation}
where $l\in\rP$ is a lift of $\bar l\in \rL_\rP$.
We define the normalized parabolic restriction as 
\[\Res^\rG_{\rL_\rP}=|\rU^F_\rP|^{-1}\mathrm{res}^{\rG}_{\rL_\rP}:C(\rG^F)\to C(\rL_\rP^F).\]

For any pair $\rP\subset \rQ\in par$ of $F$-stable standard parabolic subgroups the quotient 
\[\rB_{\rP,\rQ}:=\rP/\rU_\rQ\subset\rL_\rQ=\rQ/\rU_\rQ\]
is a $F$-stable parabolic subgroup with 
unipotent radical
$\rU_{\rP,\rQ}=\rU_\rP/\rU_\rQ$ and
Levi quoteint $\rL_\rP$.
We have the following transitivity propoerty: We have an equality
    \begin{equation}
    \mathrm{Res}^{\rL_\rQ}_{\rL_{\rP}}\circ\mathrm{Res}^\rG_{\rL_\rQ}=\mathrm{Res}^\rG_{\rL_\rP}
    \end{equation}

We recall the defintion of stable fucntions on $\rG^F$ following \cite[Section 6.5]{chen}.
Let $\hat\rG$ be the dual group of $\rG$
over $k$ introduced in \cite{DL76}. It has a
canonical Frobenious endomorphism $F:\hat\rG\to\hat\rG$ and a $F$-stable mamixal torus $\hat\rT$.
According to \cite[Section 5.6]{DL76}, the set of $F$-stable semisimple conjugacy classes of $\hat\rG$ are in bijection with the set 
$(\hat\rT//\rW)^F$ of $F$-fixed points of the GIT quotient $\hat\rT//\rW$. The main results in \cite{DL76} implies that there is a surjective map
\begin{equation}\label{DL map}
    \mathcal L:\mathrm{Irr}(\rG^F)\to(\hat\rT//\rW)^F
\end{equation}
where $\mathrm{Irr}(\rG^F)$ is the set of isomorphism classes of irreducible complex representations of $\rG^F$.
We will call the fiber $\mathcal L^{-1}(\theta)$
of $\theta\in(\hat\rT//\rW)^F$  the Deligne-Lusztig packet associated to $\theta$.

For each $f\in C(\rG^F)$
and $(\pi,V)\in\mathrm{Irr}(\rG^F)$, Schur's lemma implies that 
\[\sum_{g\in\rG^F
} f(g)\pi(g)=\gamma_f(\pi)\mathrm{Id}_V\in\mathrm{End}(V)\]
for some complex number $\gamma_f(\pi)\in\mathbb C$.

\begin{definition}
A function $f\in C(\rG^F)$ is called \emph{stable} if it satisfies the following property: for any $(\pi,V), (\pi',V')\in\mathrm{Irr}(\rG^F)$ 
we have 
\[\gamma_f(\pi)=\gamma_f(\pi')\ \ \text{if}\ \ \mathcal{L}(\pi)=\mathcal{L}(\pi').\]
We denote by $C^{st}(\rG^F)$ the space of stable functions on $\rG^F$.

\end{definition}

Let $\mathbb C[(\hat{\rT}//\rW)^F]$ be the space of complex valued functions on  $(\hat{\rT}//\rW)^F$
with multiplication given by multiplication of functions.  For any $\rP\subset\rQ\in par$, 
we have a canonical Levi decomposition 
\[\rL_\rQ\cong\rL_\rP\ltimes\rU_{\rP,\rQ}\]
where we identify $\rL_\rP$ as the standard Levi subgroup containing the maximal torus $\rT_\rQ$.
The Levi decomposition induces a
 a natural inclusion $\rW_\rP\subset\rW_\rQ$
and a 
canoincal $\rW_\rP$-equivaraint 
isomorphism 
$\rT_\rP\cong\rT_\rQ$. 
It induces a  $\rW_\rP$-equivariant 
map 
$\hat\rT_\rP\cong\hat\rT_\rQ$ compatible with 
the Frobenious endomorphism and 
we denote by 
\[\mathrm{res}^{\hat\rT_\rQ}_{\hat\rT_\rP}:\mathbb C[(\hat{\rT}_\rQ//\rW_\rQ)^F]\to \mathbb C[(\hat{\rT}_\rP//\rW_\rP)^F]\]
the map given by pull back along the natural 
map $(\hat{\rT}_\rP//\rW_\rP)^F\to(\hat{\rT}_\rQ//\rW_\rQ)^F$.

\begin{proposition}\label{main thm stable group}
    \begin{enumerate}
         \item There is an algebra isomorphism
        \[
        \mathbb C[(\hat{\rT}//\rW)^F]\cong C^{st}(\rG^F)
        \]
        sending 
        the characteristic function 
        $\mathbbm 1_\theta$ of $\theta\in(\hat{\rT}//\rW)^F$ to the 
        idempotent projector $f_\theta\in C^{st}(\rG^F)$
        for the Deligne-Lusztig packet 
        $\mathcal{L}^{-1}(\theta)$, that is,
        we have $\gamma_{f_\theta}(\pi)=1$
        if $\mathcal L(\pi)=\theta$, otherwise
        $\gamma_{f_\theta}(\pi)=0$.
        
        \item For any  $\rP\subset\rQ\in par$ and $f\in C^{st}(\rL_\rQ^F)$, we have $\mathrm{res}^{\rL_\rQ}_{\rL_\rP}(f)\in C^{st}(\rL_{\rP}^F)$ and there is a commutative diagram
        \[\xymatrix{\mathbb C[(\hat{\rT}_\rQ//\rW_\rQ)^F]\ar[r]\ar[d]^{\mathrm{res}^{\hat\rT_\rQ}_{\hat\rT_\rP}}&C^{st}(\rL_\rQ^F)\ar[d]^{\mathrm{res}^{\rL_\rQ}_{\rL_{\rP}}}\\
        \mathbb C[(\hat{\rT}_\rP//\rW_\rP)^F]\ar[r]&C^{st}(\rL_\rP^F)}.\]
      
    \end{enumerate}

\end{proposition}
\begin{proof}
    Part (1) follows from the definition of stable functions. Part (2) is proved in \cite[Proposition 4.2.2 (4)]{laumon}.
\end{proof}

We have the following key vanishing properties of stable functions.
 \begin{theorem}\label{vanishing group}
     For any $f\in C^{st}(\rG^F)$
        and any $\rP\in par$, we have 
        \begin{equation}\label{vanishing}
            \sum_{u\in\rU^F_\rP} f(xu)=0\ \ \ \text{for all}\ \ x\notin\rP^F.
        \end{equation}
 \end{theorem}
     
\begin{proof}
 Let $(s)$ denote the geometric conjugacy class of a semisimple element $s \in {\hat{\rG}}^{F}$. Let $(\rT,F)$ be an $F$ stable torus and $(\hat{\rT},F)$ it's dual. Then, $\widehat{\rT^F}\simeq {\hat{\rT}}^{F}$ and if $(\rT, \theta)$ corresponds to $s \in {\hat{\rT}}^{F}$, the geometric conjugacy class of $(\rT, \theta)$ corresponds to the geometric conjugacy class $(s)$ of $s$  in $\hat{\rG}$.  We often also denote $R^\rG_{\rT}(\theta)$ by $R^\rG_{\hat{\rT}}(s)$. The irreducible representations occuring in the Deligne-Lusztig induction $R^\rG_{\rT'}(\theta ')$ for all $(\rT',\;\theta ')$ geometrically conjugate to $(\rT,\theta)$ form the geometric Lusztig series corresponding to the $F$-stable semisimple conjugacy class $(s)$ in $\hat{\rG}$.  Now, the vector space of stable central functions is generated by elements of the form
    \begin{equation}
        f_s=\sum_{\pi \in \varepsilon (\rG^F,(s))} \pi(1) \pi 
    \end{equation}
    where $\varepsilon (\rG^F,(s))$ is the geometric Lusztig series for $(\rG,F)$ corresponding to $(s)$ (cf \cite{laumon}). So, it is enough to prove the property for such funtions.  \par
We first prove the statement for $F$-stable Borel subgroup $\rB\subset \rG$, and then generalise to $\rP \in par$.  Observe that it is enough to prove the statement for a fixed $F$-stable $\rT\subset \rB$. This is because for any other such pair $(\rT',\rB')$, $\exists g \in \rG^F$ such that $\rT'=g\rT g^{-1}$ and $\rB'=g\rB g^{-1}$, and hence $\rU_{\rB'}=g\rU_\rB g^{-1}$. Let $\rT$ be an $F$-stable maximal torus and $\rB=\rT \rU_\rB \subset \rG$ be a Borel subgroup (not necessarily $F$-stable).  Let $e^\rG_{s}$ be the central idempotent in $\overline{\Q}_l[\rG^F]$ projecting onto the  Lusztig series corresponding to $s\in {\hat{\rG}}^{F}$. The Lusztig series corresponding to $s$ contains all the irreducible representations occuring in $R^\rG_{\hat{\rT}}(s')$ for all such $(\hat{\rT},s') $where $\hat{\rT}$ is an $F$ stable maximal torus in $\hat{\rG}$ and $s' \in {\hat{\rT}}^{F}$ is $\hat{\rG}$ conjugate to $s$. Let $S_{\rT,s}$ denote the set 
    $$S_{\rT,s}=\{ \theta \in \widehat{\rT^F}| (\rT,\theta) \text{ is in the geometric conjugacy class corresponding to } (s)\}, $$ 
    and $e^\rT_s$ denote the idempotent in $\overline{\Q}_l[\rT^F]$ which projects onto this set. Then, 
\begin{equation}
    [e^\rG_s H^*_c(\tilde{X}_{\rU_\rB})e^\rT_s]
    =[H^*_c(\tilde{X}_{\rU_\rB})e^\rT_s]= [e^\rG_s H^*_c(\tilde{X}_{\rU_\rB})]
\end{equation}
as $\rG^F$-module-$\rT^F$, where $\tilde{X}_{\rU_\rB}$ is the Deligne-Lusztig variety corresponding to $(\rT,\rB)$. This is because 
\[
H^*_c(\tilde{X}_{\rU_\rB}) = \bigoplus_{\theta \in \widehat{\rT^F}}H^*_c(\tilde{X}_{\rU_\rB})_{\theta}=  \bigoplus_{\theta \in \widehat{\rT^F}}R^\rG_{\rT}(\theta)
\]

and hence we have 
\[
e^\rG_s H^*_c(\tilde{X}_{\rU_\rB})= \bigoplus_{\theta \in S_{\rT,s}}R^\rG_{\rT}(\theta)= H^*_c(\tilde{X}_{\rU_\rB})e^\rT_s= e^\rG_s H^*_c(\tilde{X}_{\rU_\rB})e^\rT_s.
\]
Now if $\rT\subset \rB$ is $F$-stable, we know that 
\begin{equation}
    H^*_c(\tilde{X}_{\rU_\rB})\simeq \overline{\Q}_l[\rG^F/\rU_\rB^F]= \overline{\Q}_l[\rG^F]e_{\rU_\rB^F},
\end{equation}
where $e_{\rU_\rB^F}=|\rU_\rB^F|^{-1}\sum_{u\in \rU_\rB^F} u$. Thus, as $\rG^F$-module-$\rT^F$, we get that 
\begin{equation}\label{egseqn}
    \left[e^\rG_s\overline{\Q}_l[\rG^F]e_{\rU_\rB^F}\right]= \left[\overline{\Q}_l[\rG^F]e^\rG_se_{\rU_\rB^F}\right]=\left[\overline{\Q}_l[\rG^F]e_{\rU_\rB^F}e^\rT_s\right]
\end{equation}
since $e^\rG_{s}$ is a central idempotent in $\overline{\Q}_l[\rG^F]$. Note that from the structure of $e^\rT_s$, we know that $e_{\rU_\rB^F}e^\rT_s= e^\rT_s e_{\rU_\rB^F}=e^\rB_s$(say), since $\rT^F \rU_\rB^F=\rU_\rB^FT^F$. Further, $e^\rB_s = \sum_{b\in \rB^F}c_b b$ for some $c_b\in \overline{\Q}_l$. 
\begin{claim}\label{cl1B}
    $e^\rG_se_{\rU_\rB^F}= e_{\rU_\rB^F}e^\rT_s=e^\rB_s$ as operators from the right on $\overline{\Q}_l[\rG^F]$. 
\end{claim}
\[
\overline{\Q}_l[\rG^F]e_{\rU_\rB^F}= \overline{\Q}_l[\rG^F/\rU_\rB^F] = \bigoplus_{\theta \in \widehat{\rT^F}}\overline{\Q}_l[\rG^F/\rU_\rB^F]_{\theta}.\]
For $\phi \in \overline{\Q}_l[\rG^F]$, $\phi e_{\rU_\rB^F} = \sum_{\theta \in \widehat{\rT^F}} \phi_\theta$, where $\phi_\theta \in \overline{\Q}_l[\rG^F/\rU_\rB^F]_{\theta}$. Now, $e^\rT_s$ acts as identity on $\overline{\Q}_l[\rG^F/\rU_\rB^F]_{\theta}$ for $\theta \in S_{\rT,s}$ and zero otherwise. So, $\phi e_{\rU_\rB^F}e^\rT_s= \sum_{\theta \in S_{\rT,s}} \phi_\theta$. On the other hand, $\phi e^\rG_se_{\rU_\rB^F}= e^\rG_s\phi e_{\rU_\rB^F}$, and $e^\rG_s$ from the left acts as identity on $\overline{\Q}_l[\rG^F/\rU_\rB^F]_{\theta}$ for $\theta \in S_{\rT,s}$ and zero otherwise. Hence, $\phi e^\rG_se_{\rU_\rB^F}= e^\rG_s\phi e_{\rU_\rB^F}= \sum_{\theta \in S_{\rT,s}} \phi_\theta$, and we have proved our claim. \par
Note that, for $g\in \rG^F$
\begin{equation}
     Tr\left(\overline{\Q}_l[\rG^F]\Big|e^\rG_se_{\rU_\rB^F}g^{-1}\right) = \sum_{u\in \rU_\rB^F}\sum_{\pi \in \varepsilon (\rG^F,(s))}\pi(1)\pi(gu)  =\sum_{u\in \rU_\rB^F}f_s(gu)  
\end{equation}
Thus, to prove our proposition, it is enough to show that the above is zero for $g\not\in \rB^F$.Using claim \ref{cl1B}, we see that for $g\not\in \rB^F$
\begin{align*}
   Tr\left(\overline{\Q}_l[\rG^F]\Big|e^\rG_se_{\rU_\rB^F}g^{-1}\right) &=   Tr\left(\overline{\Q}_l[\rG^F]\Big|e_{\rU_\rB^F}e^\rT_sg^{-1}\right)\\
   &=Tr\left(\overline{\Q}_l[\rG^F]\Big|e^\rB_sg^{-1}\right)\\
   &=\sum_{b\in \rB^F}c_b\chi_{reg}(gb^{-1})=0
\end{align*}
where $\chi_{reg}$ is the character of the regular representation of $\rG^F$. Hence, we are done in the case of Borel subgroups. \par
Observe that the main step was the fact that
    \[
     \left[e^\rG_s\overline{\Q}_l[\rG^F]e_{\rU_\rB^F}\right]= \left[\overline{\Q}_l[\rG^F]e^\rG_se_{\rU_\rB^F}\right]=\left[\overline{\Q}_l[\rG^F]e_{\rU_\rB^F}e^\rT_s\right]
    \]
    Let $(\rT, \theta)$ corresponds to $s \in {\hat{\rT}}^{F}$ and $t \in {\hat{\rT}}^{F}$. If $t \in (s)$, then we have $e^\rT_t=e^\rT_s$ and \[
     [e^\rG_s H^*_c(\tilde{X}_U)e^\rT_t]
    =[H^*_c(\tilde{X}_U)e^\rT_t]= [e^\rG_s H^*_c(\tilde{X}_U)].
    \]
    On the other hand, if $t \not \in (s)$,then we have $$[e^\rG_s H^*_c(\tilde{X}_U)e^\rT_t]=0.$$ Note further that $S_{\rT,s}$ is the set of all $\theta' \in \widehat{\rT^F}$ such that $(\rT, \theta')\leftrightarrow t \in {\hat{\rT}}^{F}$ for some $t \in (s)$. 
    \par
     Now, we move onto the case of general $\rP \in par$. Let $ \rP\subset \rG $ be a parabolic subgroup (not necessarily $F$ stable), $\rP= \rU_\rP\rtimes \rL_{\rP} $ for $F$ stable Levi subgroup  $\rL_{\rP}$ containing an $F$-stable torus $\rT$ and unipotent radical $\rU_\rP=R_u(\rP)$. Note that the Deligne Lusztig varieties and induction is defined for $\rL_{\rP}^F$, and we have a similar inuction functor $R^\rG_\rL$ (cf for example chapter 9 in \cite{digne}). Let $\tilde{X}_{\rU_\rP}$ denote the Deligne Lusztig variety for $\rP$ and $H^*_c(\tilde{X}_{\rU_\rP})$ is now a $\rG^F$-module-$\rL_{\rP}^F$. Let $\hat{\rL}_{\rP}$ denote the dual of $\rL_{\rP}$. Any semisimple class $(t)$ in $\hat{\rL}_{\rP}$ gives rise to a semisimple conjugacy class $(s)$ in $\hat{\rG}$, and the map $(t) \mapsto (s)$ is finite to one. Let $e^\rL_t$ be defined similarly to $e^\rG_s$, since $\rL_{\rP}\supset \rT$ is a reductive group. Using the ideas in \cite{dat} section 2.1.4 and the fact that the Deligne-Lusztig induction is transitive we see that 
 \begin{equation}
      [e^\rG_s H^*_c(\tilde{X}_{\rU_\rP})e^\rL_t]
    =[H^*_c(\tilde{X}_{\rU_\rP})e^\rL_t]
 \end{equation}
We also have $[e^\rG_s H^*_c(\tilde{X}_{\rU_\rP})e^\rL_{t'}]=0$ if $t' \not\in (s)$. So, if we define $e^\rL_s := \sum_{(t)\mapsto (s)} e^\rL_t$, we see that 
 \begin{equation}
      [e^\rG_s H^*_c(\tilde{X}_{\rU_\rP})e^\rL_s]
    =[H^*_c(\tilde{X}_{\rU_\rP})e^\rL_s]=  [e^\rG_s H^*_c(\tilde{X}_{\rU_\rP})]
 \end{equation}
 Coming back to our particular case when $\rP$ is $F$-stable, we again have 
 \[
  H^*_c(\tilde{X}_{\rU_\rP})\simeq \overline{\Q}_l[\rG^F/\rU_\rP^F]= \overline{\Q}_l[\rG^F]e_{\rU_\rP^F}
 \]
 where everything is defined similarly to the Borel case. Defining $e^\rP_s:=e^\rL_se_{\rU_\rP^F}=e_{\rU_\rP^F}e^\rL_s$ in the same way, and using ideas in claim \ref{cl1B}, we get
 \begin{equation}
     e^\rG_se_{\rU_\rP^F}= e_{\rU_\rP^F}e^\rL_s=e^\rP_s
 \end{equation}
 which was the main step in the proof of the Borel case. Hence, following the same steps as in that proof, we get that $\sum_{u\in \rU_\rP^F}f_s(gu)= 0$ if $g\not\in \rP^F$ which finishes the proof. 
\end{proof}

\subsection{Stable functions on finite Lie algebras}
 We write $\mathfrak g$, $\mathfrak b$, $\mathfrak t$, $\mathfrak p$, $\mathfrak l_\rP$, $\mathfrak n_\rP$ for the Lie algebras of $\rG$, $\rB$, $\rT$, $\rP$, $\rL_\rP$, $\rU_\rP$.
We denote by $\mathfrak g^F,\mathfrak b^F$, etc the $F$-fixed points of $\mathfrak g$, $\mathfrak b$, etc.

The group $\rG^F$ acts on $\fg^F$ via the adjoint representaion and we denote by $C(\fg^F)$ the space of 
$\rG^F$-invaraint functions on $\fg^F$.
We equip $C(\fg^F)$ with the convolution prodcut
\[f\star f'(X)=|\fg^F|^{-1/2}\sum_{Y\in\fg^F}f(X-Y)f'(Y
).\]

For any $\rP\in par$, we have the normalized parabolic restriction map $\mathrm{Res}^\fg_{\fl_\rP}:C(\fg^F)\to C(\fl_\rP^F)$ defined by 
\[\mathrm{Res}^\fg_{\fl_\rP}(f)(\bar v)=|\fn_\rP^F|^{-1}\sum_{n\in\fn_\rP^F}f(v+n)\]
where $v\in\fp^F$ is a lift of $\bar l$. 
For any $\rP\subset\rQ\in par$, we have the following transitivity property
  \begin{equation}
    \mathrm{Res}^{\fl_\rQ}_{\fl_{\rP}}\circ\mathrm{Res}^\fg_{\fl_\rQ}=\mathrm{Res}^\fg_{\fl_\rP}.
    \end{equation}

Let $V$ be a finite dimensional $\mathbb F_q$-vector 
space and let $V^*$ be its dual vector space.
Denote by $\mathbb C[V]$ (resp. $\mathbb C[V^*]$) the space of 
complex valued functions on $V$ (resp. $V^*$).
We have the Fourier-transform
$\mathrm{FT}_V:\mathbb C[V]\to \mathbb C[V^*]$
defined by the formula
\begin{equation}\label{FT_V}
\mathrm{FT}_V(f)(X^*)=|V|^{-1/2}\sum_{Y\in V}\psi(( X^*(Y))f(Y).
\end{equation}

In this paper
we are mainly interested in the case $V=\fg^F$.
The assumption on $p$ implies that there exists a $\rG$-invariant non-degenerate bilinear form 
\[\langle-,-\rangle:\fg\times\fg\to k\]
which is defined over $\mathbbm F_q$. 
In this section we will
fix such an invariant form. It induces a $\rG^F$-invariant isomorphism
$\fg^F\cong(\fg^*)^F$  and the Fourier-transform~\eqref{FT_V} restricts to a 
self map 
$\mathrm{FT}_\fg:C(\fg^F)\to 
C((\fg^*)^F)\cong C(\fg^F)$
given by the formula
\begin{equation}
\mathrm{FT}_\fg(f)(X)=|\fg^F|^{-1/2}\sum_{Y\in\fg^F}\psi(\langle X,Y\rangle)f(Y).
\end{equation}

We have the following well-known properties of Fourier transforms, see, e.g., \cite{Le, Lu1}:
Note that the restriction of $\langle-,-\rangle$ 
to $\fp\times\fp$ descends to
a $\rL_\rP$-invariant non-degenerate bilinear form
$\langle-,-\rangle_{\fl_\rP}:\fl_\rP\times\fl_\rP\to k$
on $\fl_\rP=\fp/\fn_\rP$ and we write 
$\mathrm{FT}_{\fl_\rP}:C(\fl_\rP^F)\to C(\fl_\rP^F)$
the corresponding Fourier transform.

\begin{lemma}\label{properties of FT}
\begin{enumerate}
    \item 
    $\mathrm{FT}_\fg(f\star f')=\mathrm{FT}_\fg(f)\mathrm{FT}_\fg(f')$
    \item
    $\FT_\fg(ff')=\mathrm{FT}_\fg(f)\star\mathrm{FT}_\fg(f')$
    \item 
    $\FT_\fg^2(f)=f^-$, where $f^-(X)=f(-X)$.
    \item 
    $\FT_{\fl_\rP}\circ\mathrm{Res}^\fg_{\fl_\rP}=\mathrm{Res}^\fg_{\fl_\rP}\circ\FT_\fg$.
   \end{enumerate}
\end{lemma}

Consider the Chevalley map
$\chi:\fg\to \fg//\rG\cong\ft//\rW$.

\begin{definition}\label{stable function on Lie}
    A function $f\in C(\fg^F)$ is called \emph{stable}
    if its Fourier transform 
    $\mathrm{FT}_\fg(f):\fg^F\to\mathbb C$
    is constant on the fiber $\chi^{-1}(\theta)$ 
    for all $\theta\in(\ft//\rW)^F$

\end{definition}

\begin{remark}
The Fourier transform induces a 
 decomposition $C(\fg^F)=\oplus_{\theta\in(\ft//\rW)^F} C(\fg^F)_\theta$
 where $C(\fg^F)_\theta$ is the subspace of functions $h\in C(\fg^F)$ such that 
 $\mathrm{FT}_\fg(h)$ is supported on $\chi^{-1}(\theta)$. One can view 
 $C(\fg^F)_\theta$ as Lie algebra analogue of Deligne-Lusztig packets and a function $f$ is stable if and only if for any $\theta\in(\ft//\rW)^F$ there exists a constant $c_\theta$ such that 
 \[f\star h=c_\theta\cdot h\]
 for all $h\in C(\fg^F)_\theta$.
    
\end{remark}

Let $\mathbb C[(\ft//\rW)^F]$ be the space of complex valued functions on  $(\ft//\rW)^F$
with multiplication given by multiplication of functions.  For any $\rP\subset\rQ\in par$, 
we have a 
canonical $\rW_\rP$-equivariant 
isomorphism 
$\ft_\rP\cong\ft_\rQ$
compatible with 
the Frobenious endomorphism and 
we denote by 
\[\mathrm{res}^{\ft_\rQ}_{\ft_\rP}:\mathbb C[(\ft_\rQ//\rW_\rQ)^F]\to \mathbb C[(\ft_\rP//\rW_\rP)^F]\]
the map given by pull back along the natural 
map $(\ft_\rP//\rW_\rP)^F\to(\ft_\rQ//\rW_\rQ)^F$

We haves the following properties of stable functions on $\fg^F$:
\begin{theorem}\label{main thm stable Lie}
\begin{enumerate}
    \item 
    For any $z\in\mathbb C
    [(\ft//\rW)^F]$ the composition
    \[f_z:=\FT_\fg\circ(\chi^*(z)^-):\fg^F\to(\ft//\rW)^F\to\mathbb C\]
    is a stable function on $\fg^F$ and 
    the assignment 
    $z\to f_z$
    defines an algebra isomorphism
    \[\mathbb C[(\ft//\rW)^F]\cong C^{st}(\fg^F).\]
    
    \item For any  $\rP\subset\rQ\in par$ and $f\in C^{st}(\fl_\rQ^F)$, we have $\mathrm{Res}^{\fl_\rQ}_{\fl_\rP}(f)\in C^{st}(\fl_{\rP}^F)$ and there is a commutative diagram
        \[\xymatrix{\mathbb C[(\ft_\rQ//\rW_\rQ)^F]\ar[r]\ar[d]^{\mathrm{res}^{\ft_\rQ}_{\ft_\rP}}&C^{st}(\fl_\rQ^F)\ar[d]^{\mathrm{Res}^{\fl_\rQ}_{\fl_{\rP}}}\\
        \mathbb C[(\ft_\rP//\rW_\rP)^F]\ar[r]&C^{st}(\fl_\rP^F)}.\]
where the horizontal maps are the isomorphisms in part (1)
        \item 
        For any $f\in C^{st}(\fg^F)$
        and any $\rP\in par$, we have 
        \begin{equation}\label{vanishing Lie algebra}
            \sum_{n\in\fn^F_\rP} f(X+n)=0\ \ \ \text{for all}\ \ X\notin\fp^F.
        \end{equation}
\end{enumerate}
    
\end{theorem}
\begin{proof}
Part (1) follows  from
the definition and Properties of  Fourier transforms in Lemma \ref{properties of FT}.

Proof of (2). Without loss of generality, we can assume $\rQ=G$.
We claim that, 
for any $z\in\mathbb C[(\ft//\rW)^F]$,
we claim that 
\[\Res^\fg_{\fl_\rP}(\chi^*(z)^-)=\chi^*_{\fl_\rP}(\mathrm{res}^\ft_{\ft_\rP}(z))^-.\]
Then Lemma \ref{properties of FT} implies 
\[\Res^\fg_{\fl_\rP}(\FT_\fg\circ(\chi^*(z)^-))=
\FT_{\fl_\rP}(\Res^\fg_{\fl_\rP}(\chi^*(z)^-))=\FT_{\fl_\rP}\circ(\chi^*_{\fl_\rP}(\res^\ft_{\ft_\rP}(z))^-)\]
Part (2) follows. To proof the claim note that we have following commutative diagram. 
\begin{equation}\label{Chevalley maps}
\xymatrix{\fl_\rP\ar[d]^{\chi_{\fl_\rP}}&\fp\ar[r]^{i_\rP}\ar[l]_{\pi_\fp}&\fg\ar[d]^\chi\\
 \ft_\rP//\rW_\rP\ar[rr]&&\ft//\rW}
\end{equation}
It follows that 
\[\Res^\fg_{\fl_\rP}(\chi^*(z)^-)(\bar v)=|\fn_\rP^F|^{-1}\sum_{n\in\fn_\rP^F} \chi^*z(-v-n)=
|\fn_\rP^F|^{-1}\sum_{n\in\fn_\rP^F} (\chi_{\fl_\rP}\circ\pi_\rP)^*\res^\ft_{\ft_\rP}(z)(-v-n)=\]
\[=|\fn_\rP^F|^{-1}(|\fn_\rP^F|^{}\chi_{\fl_\rP}^*\res^\ft_{\ft_\rP}(z)(-\bar v))=\chi_{\fl_\rP}^*\res^\ft_{\ft_\rP}(z)^-(\bar v).
\]
The proof is completed.

Proof of (3). 
We have the following Cartesian diagram
\[\xymatrix{\fp\ar[r]^{i_\rP}\ar[d]^{\pi_\rP}&\fg\cong\fg^*\ar[d]^{i_\rP^*}\\
\fl_\rP^*\cong\fl_\rP=\fp/\fn_\rP\ar[r]^{\pi_\rP^*}&\fg/\fn_\rP\cong\fp^*}\]
where the isomorphisms are induced by the  
invariant $\langle-,-\rangle$ form on $\fg$ and the maps are natural inclusions and quotients.
We need to show that the function
$h:\fg/\fn_\rP\cong\fp^*\to\mathbb C$  given by 
\[h(v)=\sum_{w\in(i_\rP^*)^{-1}(v)} f(w)\]
is supported on $\fl_\rP^F$.
Direct computations
of Fourier transforms show that (a) 
$h$ is supported on $\fl^F_\rP$
if and only if its Fourier transform 
$\mathrm{FT}_{\fp^*}(h):\fp\to\mathbb C$
 is constant 
on the fibers of $\pi_\rP:\fp\to\fl_\rP$
(b) $\FT_\fg(f)|_{\fp^F}=\mathrm{FT}_{\fp^*}(h)$. Thus we reduce to show that  the restriction 
$\FT_\fg(f)|_{\fp^F}:\fp^F\to\mathbb C$ 
to the subspace $\fp^F\subset\fg^F$
is constant on the fibers of the projecrion $\fp^F\to\fl^F_\rP$. This follows from Part (1)
and diagram~\eqref{Chevalley maps}
\end{proof}

\section{From stable functions to Bernstein centers}\label{stable to center}
For the rest of the paper,  we will fix
a non-degenerate $G(\mathbbm k)$ invariant bilinear form $\langle-,-\rangle$ on the Lie algebra $\mathrm{Lie}(G)(\mathbbm k)$ such that for any $P\in Par$, we have 
$\mathrm{Lie}(P)^{\perp}=\mathrm{Lie}(P^+)$.
Here $\mathrm{Lie}(P)^{\perp}=\{X\in\mathrm{Lie}(G)(\mathbbm k)|\langle X,\mathrm{Lie}(P)\rangle\subset\omega\mathcal{O}\}$.

\subsection{Depth-zero case}
For any $P\in Par$ we have
$P/P^+\cong\rG_P^F$ where $\rG_P$ is a $k$-connected reductive group defined over $\mathbb F_q$.  For any $P\subset Q\in Par$,
we have
\[\rU^F_{P,Q}:=P^+/Q^+\subset \rB^F_{P,Q}:=P/Q^+\] where $\rB_{P,Q}\subset\rG_Q$
is a $F$-stable parabolic subgroup with unipotent radical $\rU_{P,Q}$ and Levi quotient $\rB_{P,Q}/\rU_{P,Q}\cong\rG_P$.
We will write $\rB_P=\rB_{I,P}$ and $\rU_P=\rU_{I,P}$, etc.

Let $T(\mathcal O)$ be the maximal compact subgroup of the split maximal torus $T\subset G$. We have $T(\mathcal O)\subset I$ and for any $P\in Par$ we have 
$\text{Im}(T(\mathcal O)\to I\to P\to\rG_P^F)=\rT_P^F$ where 
$\rT_P\subset\rB_P\subset\rG_P$ is a maximal $F$-stable torus. We denote by $\rW_P=N(\rT_P)/\rT_P$ the Weyl group. 
When $P=G(\mathcal O)$, we will write 
$\rG_P=\rG$, $\rB=\rB_P$, $\rT=\rT_P$, and 
$\rW=\rW_P$.
Note that we have a $F$-equivariant isomorphism 
$\rT\cong\rT_P$ such that the natural map
$T(\mathcal O)\to\rT_P^F$ factors as 
$T(\mathcal O)\to\rT^F\cong\rT^F_P$

Let $\rW_a=N(T(F))/T(\mathcal O)$ be the affine Weyl group of $G(\mathbbm k)$. We have a natural surjection 
$\rW_a\to N(T(F))/T(F)\cong\rW$. For any $P\in Par$, the Weyl group $\rW_P$
is naturally a subgroup  $\rW_P\subset\rW_a$ and the isomorphism 
$\rT\cong\rT_P$ is $\rW_P$-equivariant 
where $\rW_P$ acts on $\rT$ through the morphism $\rW_P\to\rW_a\to\rW$.

For any $P\in Par$, we consider the normalized 
convolution $\star^\mu$ on $C(\rG^F_P)$ given by
\[f\star^\mu f'=\mu(P^+)f\star f'\]
with multiplicative unit 
$\mu(P^+)^{-1}\mathrm 1_e\in C(\rG_P^F)$.
We have a natural algebra inclusion 
\[\iota_P:(C(\rG^F_P),\star^\mu)\to  
 (C^{\infty}_c(\frac{P/P^+}{P}),*)\to  (C^{\infty}_c(\frac{G(\mathbbm k)/P^+}{P}),*)\]
sending $f_{\rG_P}:\rG^F_P\to\mathbb C$
to the composed map
$f_P:P\to P/P^+\cong\rG_P^F\stackrel{f_{\rG_P}}\to\mathbb C$.
\begin{lemma}\label{comp}
   For any $P\subset Q\in Par$, there is a commutative diagram of algebra homomorphisms
   \[\xymatrix{(C^{st}(\rG_Q^F),\star^\mu)\ar[r]^{\iota_Q}\ar[d]^{\Res^{\rG_Q}_{\rG_P}}&(C^{\infty}_c(\frac{G(\mathbbm k)/Q^+}{Q}),*)\ar[d]^{*\delta_{P^+}}\\
   (C^{st}(\rG_P^F),\star^\mu)\ar[r]^{\iota_P
   }&(C^{\infty}_c(\frac{G(\mathbbm k)/P^+}{P}),*)}\]
\end{lemma}
\begin{proof}
For any $f_{\rG_Q}\in C^{st}(\rG_Q^F)$ and  $x\in Q$ with image $\bar x\in Q/Q^+\cong\rG_Q^F$, we have
\[\iota_Q(f_{\rG_Q})*\delta_{P^+}(x)=\int_{G(\mathbbm k)}f_Q(xy^{-1})\delta_{P^+}(y)d\mu(y)=\mu(P^+)^{-1}\int_{P^+}f_Q(xy^{-1})d\mu(y)=\]
\[=\sum_{\bar u\in P^+/Q^+}\int_{Q^+} f_Q(xuz)d\mu(z)=
\mu(Q^+)\mu(P^+)^{-1
}(\sum_{\bar u\in P^+/Q^+}f_{\rG_Q}(\bar x\bar u))=|\rU_{P,Q}^F|^{-1}(\sum_{\bar u\in P^+/Q^+}f_{\rG_Q}(\bar x\bar u)).
\]
Since $f_{\rG_Q}$
is stable 
we have 
$|\rU_{P,Q}^F|^{-1}\sum_{\bar u\in P^+/Q^+}f_{\rG_Q}(\bar x\bar u)=0$
if $\bar x\notin P/Q^+\cong\rB_{P,Q}^F\subset\rG_Q^F$
and is equal to 
$\Res^{\rG_Q}_{\rG_P}(f_{\rG_Q})(\bar x)$ if $\bar x\in\rB_{P,Q}^F$.
All together, we obtain that $\iota_Q(f_{\rG_Q})*\delta_{P^+}$
is supported on $P$
and its value at
$x\in P$ is equal to 
\[\iota_Q(f_{\rG_Q})*\delta_{P^+}(x)=\Res^{\rG_Q}_{\rG_P}(f_{\rG_Q})(\bar x)=\iota_P(\Res^{\rG_Q}_{\rG_P}(f_{\rG_Q}))(x).\]
The compatibility with multiplication is clear.
The proof is completed.
\end{proof}

\begin{remark}
    Note that in Lemma \ref{comp} it is essential to consider stable functions $C^{st}(\rG_P^F)$
    instead of  all class functions $C(\rG_P^F)$. 
\end{remark}
\begin{lemma}\label{j_P}
    For any $P\in Par$, there is an algebra homomorphism 
    \[j_P:(C^{st}(\rG^F),\star)\to (C^{st}(\rG_P^F),\star^\mu)\]
    such that for any $P\subset Q\in Par$ we have the following commutative diagram
    \[\xymatrix{C^{st}(\rG^F)\ar[r]^{j_Q}\ar[dr]^{j_P}&C^{st}(\rG_Q^F)\ar[d]^{\Res^{\rG_Q}_{\rG_P
    }}\\
    &C^{st}(\rG_P^F)}\]
\end{lemma}
\begin{proof}
Note that the 
natural $\rW_P$-equivaraint map
$\rT\cong\rT_P$
gives rise to a
$\rW_P$-equivariant map
$\hat\rT\cong\hat\rT_P$
compatible with the Frobenius endomorphisms.
It induces a map 
 \begin{equation}\label{t_P}
     t_P:(\hat\rT//\rW_P)^F\to(\hat\rT//\rW)^F
 \end{equation}
 between the sets  $F$-stable semisimple conjugacy classes of $\rG_P$
 and $\rG$
 and, by applying Theorem \ref{main thm stable group},
 we define the desired algebra homomorpohism as 
 \begin{equation}\label{eqj_P}
     j_P:C^{st}(\rG^F)\cong\mathbb C[(\hat\rT//\rW)^F]\stackrel{\mu(P^+)^{-1}t_P^*
 }\longrightarrow\mathbb C[(\hat\rT//\rW_P)^F]\cong C^{st}(\rG_P^F).
 \end{equation}
The compatibility with parabolic restrctions follows from part 2 of 
 Theorem \ref{main thm stable group}.   
\end{proof}

\begin{theorem}\label{stable group to center}
    There is an algebra homomorphism 
    \[\xi^0:C^{st}(\rG^F)\to\mathcal Z^{0}(G)\]
    such that for any depth-zero representation $(\pi,V)$ and a vector $v\in V^{P^+}$ we have 
    \begin{equation}\label{formula group}
        \xi^0(f)(v)=\mu(P^+)\sum_{x\in\rG_P^F}j_P(f)(x)\pi^{P^+}(x)v
    \end{equation}
    where $j_P:C^{st}(\rG^F)\to C^{st}(\rG_P^F)$ is the map in Lemma \ref{j_P}
    and 
    $\pi^{P^+}$
    denotes the natrual representation of $\rG_P^F$
    on $V^{P^+}$.
\end{theorem}
\begin{proof}
    Lemma \ref{comp}, Lemma \ref{j_P}, and Theorem \ref{main thm depth 0}
    together
    gives rise to an algebra homomorphsim
    \[\xi^0:C^{st}(\rG^F)\stackrel{\lim_{P\in Par} j_P}\longrightarrow\lim_{P\in Par} C^{st}(\rG_P^F)\stackrel{\lim_{P\in  Par} \iota_P}\longrightarrow\lim_{P\in Par}\mathcal M_P\cong \mathcal Z^0(G)\]
    where the limit  $\lim_{P\in Par} C^{st}(\rG_P^F)$
    is taking respect to the parabolic restriction maps.
    Formula~\eqref{formula group} follows from Proposition \ref{eval}.
    The proof is completed.
    
\end{proof}

\subsection{Positive integral depths case}
For any $P\in Par$ and $r\in\mathbb Z_{>0}$, we have 
$P_r/P_r^+\cong\fg_{P_r}^F$ where $\fg_{P_r}$ is a $k$-vector space defined over $\mathbb F_q$. There is natural action of $\rG_P$
on $\fg_{P_r}$ defined over $\mathbb F_q$.
For any $P\subset Q\in Par$,
we have   subgroups $\fn_{P_r,Q_r}^F:=P^+_r/Q^+_r\subset\mathfrak b_{P_r,Q_r}^F:=P_r/Q^+_r$, where 
$\mathfrak n_{P_r,Q_r}\subset\fb_{P_r,Q_r}\subset\fg_{Q_r}$
are $k$-subspaces define over $\mathbb F_q$.
We have a natural isomorphism
$\fb_{P_r,Q_r}/\fn_{P_r,Q_r}\cong\fg_{P_r}$.
We will write $\fn_{P_r}=\fn_{I_r,P_r}, \fb_{P_r}=\fb_{I_r,P_r}$, etc, and 
when $P=G(\mathcal O)$ we will write 
$\fg_r=\fg_{P_r}, \fn_r=\fn_{P_r}$, etc.

We have $T(\mathcal O)_r/T(\mathcal O)_r^+\cong\ft_r^F$ where $\ft_r\subset\fg_r$ is a $k$-subspace over $\mathbb F_q$.
For any $P\in Par$ there exists a $k$-space 
$\ft_{P_r}\subset\fg_{P_r}$ and an 
isomorphism $\ft_r\cong\ft_{P_r}$
such that natural map 
$T(\mathcal O)_r\to P_r\to P_r/P_r^+\cong\fg^F_{P_r}$ factors as 
$T(\mathcal O)_r\to\ft_r^F\cong\ft_{P_r}^F\subset\fg_{P_r}^F$.

Let $C(\fg_{P_r}^F)$ be the space of $\rG_P^F$-invaraint functionas on $\fg_{P_r}^F$.
Usuing the subspaces $\fb_{P_r,Q_r}$, etc, we can similarly defined the subspace 
$C^{st}(\fg^F_{P_r})$ of stable functions on $\fg_{P_r}^F$ and parabolic restrcition map
\[\Res^{\fg_{Q_r}}_{\fg_{P_r}}:C(\fg_{Q_r}^F)\to C(\fg_{P_r}^F).\]

\begin{lemma}\label{constant}
    For any $P\subset Q\in Par$ and $r\in\mathbb Z_{>0}$, we have 
$\mu(P^+_r)|\fg_{P_r}^F|^{1/2}=\mu(Q^+_r)|\fg_{Q_r}^F|^{1/2}$.
We will write $c_{\mu,r}=\mu(P^+_r)|\fg_{P_r}^F|^{1/2}$ for the constant. 
\end{lemma}
\begin{proof}
    Indeed, we have 
    $(\mu(P^+_r)/\mu(Q^+_r))^2=|\fn_{P_r,Q_r}^F|^2=|\fg_{Q_r}^F|/|\fg_{P_r}^F|$.
    \end{proof}

Consider the following normalized convolution product on $C(\fg_{P_r}^F)$:
\[f\star^\mu f(X)=
c_{\mu,r}(f\star f')(X)=
\mu(P^+_r)\sum_{Y\in\fg_{P_r}^F} f(X-Y)f'(Y)\]
with multiplicative unit 
$\mu(P^+_r)^{-1}\mathrm 1_e\in C(\fg_{P_r}^F)$.
We have a natural algebra inclusion 
\[\iota_{P_r}:(C(\fg^F_{P_r}),\star^\mu)\to  
 (C^{\infty}_c(\frac{P_r/P_r^+}{P}),*)\to  (C^{\infty}_c(\frac{G(\mathbbm k)/P_r^+}{P}),*)\]
sending $f_{\fg_{P_r}}:\fg^F_{P_r}\to\mathbb C$
to the composed map
$f_{P_r}:P_r\to P_r/P^+_r\cong\fg_{P_r}^F\stackrel{f_{\fg_{P_r}}}\to\mathbb C$.
\begin{lemma}\label{comp Lie}
   For any $P\subset Q\in Par$, there is a commutative diagram of algebra homomorphisms
   \[\xymatrix{(C^{st}(\fg_{Q_r}^F),\star^\mu)\ar[r]^{\iota_{Q_r}}\ar[d]^{\Res^{\fg_{Q_r}}_{\fg_{P_r}}}&(C^{\infty}_c(\frac{G(\mathbbm k)/Q_r^+}{Q}),*)\ar[d]^{*\delta_{P^+_r}}\\
   (C^{st}(\fg_{P_r}^F),\star^\mu)\ar[r]^{\iota_{P_r}
   }&(C^{\infty}_c(\frac{G(\mathbbm k)/P_r^+}{P}),*)}\]
\end{lemma}
\begin{proof}
Same proof as in the depth-zero case.
\end{proof}

Let $\fg_P=\mathrm{Lie}(\rG_P)$
be the Lie algebra of $\rG_P$
equipped with the adjoint representation of $\rG_P$.
A choice of uniformizer $\omega\in\mathcal O_{\mathbbm k}$ induces a
$\rG_P$-equivariant isomorphism
\begin{equation}\label{identification}
   \omega^r:\fg_{P}\cong\fg_{P_r} 
\end{equation}
sending 
$\fn_{P,Q}, \fb_{P,Q}, \ft_{P}$, etc to $\fn_{P_r,Q_r}, \fb_{P_r,Q_r}, \ft_{P_r}$, etc.
The choice of the invariant form on $\mathrm{Lie}(G)(\mathbbm k)$
induces invariant forms 
on $\fg_P\cong\mathrm{Lie}(P)/\mathrm{Lie}(P^+)$ denoted by 
$\langle-,-\rangle_P$. Using the above invaraint isomorphism~\eqref{identification}, we get an invariant form on
$\fg_{P_r}$  
\begin{equation}\label{invariant form}
    \langle-,-\rangle_{P_r}:\fg_{P_r}\times\fg_{P_r}
    \to k.
\end{equation}

Let $\mathrm{FT}_{P_r}:C(\fg^F_{P_r})
\to C(\fg^F_{P_r})$ be the Fourier transform on $C(\fg^F_{P_r})$ associated to $\langle-,-\rangle_{P_r}$.

\begin{lemma}\label{j_P_r}
    For any $P\in Par$, there is an algebra homomorphism 
    \[j_{P_r}:(C^{st}(\fg_r^F),\star)\to (C^{st}(\fg_{P_r}^F),\star^\mu)\]
    such that for any $P\subset Q\in Par$ we have the following commutative diagram
    \[\xymatrix{C^{st}(\fg_r^F)\ar[r]^{j_{Q_r}}\ar[dr]^{j_{P_r}}&C^{st}(\fg_{Q_r}^F)\ar[d]^{\Res^{\fg_{Q_r}}_{\fg_{P_r}
    }}\\
    &C^{st}(\fg_{P_r}^F)}\]
\end{lemma}
\begin{proof}
For any $P\in Par$, we have the Chevalley map
$\chi_{P_r}:\fg_{P_r}\to \fg_{P_r}//\rG_P\cong\ft_{P_r}//\rW_P$ and using the identification~\eqref{identification} the same proof of Theorem \ref{main thm stable Lie} implies there is an algebra isomorphism
\begin{align*}
    \mathbb C[(\ft_{P_r}//\rW_P)^F]&\xlongrightarrow{\sim} C^{st}(\fg_{P_r}^F)\\
    z&\longmapsto f_z:= \mathrm{FT}_{P_r}((\chi_{P_r}^*z)^-)
\end{align*}
compatible with parabolic restriction maps.
Note that the 
natural $\rW_P$-equivaraint map
$\ft_r\cong\ft_{P_r}$
compatible with the Frobenius endomorphisms.
It induces a map 
 \begin{equation}\label{t_P_r}
     t_{P_r}:(\ft_{P_r}//\rW_P)^F\to(\ft_{r}//\rW)^F
 \end{equation}
 and, by applying Theorem \ref{main thm stable Lie},
 we define the desired algebra homomorphism as 
 \begin{equation}\label{j_{P_r}}
     j_{P_r}:C^{st}(\fg_r^F)\cong\mathbb C[(\ft_r//\rW)^F]\stackrel{c_{\mu,r}^{-1}t_{P_r}^*
 }\longrightarrow\mathbb C[(\ft_{P_r}//\rW_P)^F]\cong C^{st}(\fg_{P_r}^F)
 \end{equation}
 where $c_{\mu,r}$ is the constant in Lemma \ref{constant}. 
\end{proof}

\begin{theorem}\label{stable lie to center}
    There is an algebra homomorphism 
    \[\xi^r:C^{st}(\fg_r^F)\to\mathcal Z^{r}(G)\]
    such that for any depth  $\leq r$ representation $(\pi,V)$ and a vector $v\in V^{P_r^+}$ we have 
    \begin{equation}\label{formula Lie}
        \xi^r(f)(v)=\mu(P^+_r)\sum_{x\in\fg_{P_r}^F}j_{P_r}(f)(x)\pi^{P_r^+}(x)v
    \end{equation}
    where $j_{P_r}:C^{st}(\fg_r^F)\to C^{st}(\fg_{P_r}^F)$ is the map in Lemma \ref{j_P_r}
    and 
    $\pi^{P_r^+}$
    denotes the natural representation of $\fg_{P_r}^F$
    on $V^{P^+}$.
\end{theorem}
\begin{proof}
    Lemma \ref{comp Lie}, Lemma \ref{j_P_r}, and Theorem \ref{main thm depth r}
    together
    gives rise to an algebra homomorphism
    \[\xi^r:C^{st}(\fg_r^F)\stackrel{\lim_{P\in Par} j_{P_r}}\longrightarrow\lim_{P\in Par} C^{st}(\fg_{P_r}^F)\stackrel{\lim_{P\in  Par} \iota_{P_r}}\longrightarrow\lim_{P\in Par}\mathcal M_{P_r}\cong\mathcal Z^r(G).\]
    Formula~\eqref{formula Lie} follows from Proposition \ref{evalr}.
    The proof is completed.
    
\end{proof}

\subsection{Depth-$r$ Deligne-Lusztig parameters }\label{DL}
Theorem \ref{main thm stable group}
and 
Theorem \ref{main thm stable Lie}
imply that 
there exists natural bijections between the sets of algebra homomorphisms
$\mathrm{Hom}(C^{st}(\rG^F),\mathbb C)\cong\mathrm{Hom}(\mathbb C[(\hat\rT//\rW)^F)],\mathbb C)$ and 
$\mathrm{Hom}(C^{st}(\fg_r^F),\mathbb C)\cong\mathrm{Hom}(\mathbb C[(\ft_r//\rW)^F)],\mathbb C)$
with the sets $(\hat\rT//\rW)^F$   
of 
$F$-stable semisimple conjugacy classes in 
$\hat\rG$ and  $F$-stable depth-$r$ semisimple conjugacy classes in $\fg_r$ respectively.  

 For each  irreducible representation $(\pi,V)$  of depth-$r$, by 
composing 
the maps $\xi^r$ in 
Theorem \ref{stable group to center} and Theorem \ref{stable lie to center}
with the evaluation map $\mathcal Z^r(G)\to\mathrm{End}(\pi)=\mathbb C$, we obtain maps
\[C^{st}(\rG^F)\stackrel{\xi^0}\to\ZZ^0(G)\to\mathrm{End}(\pi)=\mathbb C\]
\[C^{st}(\fg_r^F)\stackrel{\xi^r}\to\ZZ^r(G)\to\mathrm{End}(\pi)=\mathbb C\]
and hence an element 
$\theta(\pi)$ in
$(\hat\rT//\rW)^F$
if $r=0$
or in $(\ft_r//\rW)^F$ if $r>0$. 
We will call $\theta(\pi)$
the 
\emph{depth-r Deligne-Lusztig parameter} of $\pi$,
\begin{remark}
    For most positive rational depths ($r \in \Z_{(p)} \cap \Q_{>0}$), similar construction of depth-$r$ Deligne-Lusztig parameter attached to depth-$r$ irreducible representations is done in our later work \cite[Section 6.1]{bcrationaldepths}. This enabled us to attach Deligne-Lusztig parameters to any irreducible representation in a similar way if $p$ does not divide the order of the Weyl group of $G$. 
\end{remark}
\subsection{
Minimal K-types}\label{K-types}
A depth-$r$ (unrefined) minimal $K$-type
of $G(\mathbbm k)$ is a pair $(P_r,\chi)$ where $P_r\subset P$ is the $r$-th congurence subgroup of 
$P\in Par$
and $\chi$ is a representation of $P_r/P_r^+$
such that (1) if $r=0$ then 
$\chi$ is an irreducible cuspidal representation of $\rG_P^F\cong P/P^+$
(2) if $r>0$, then $\chi$
is a non-degenerate character of $\fg_{P_r}^F\cong P_r/P_r^+$.

Using the  additive character 
$\psi$ one can identify
characters of $\fg_{P_r}^F$, $r>0$,
with elements in $ (\fg_{P_r}^*)^F=\mathrm{Hom}(\fg_{P_r}^F,\mathbb F_q)$
and a character $\chi\in(\fg_{P_r}^*)^F$
is called non-degenerate
if the clousre of the 
$\rG_P$-orbit of $\chi$ in $\fg_{P_r}^*$
does not contains the zero vector.

Let $(\pi,V)$ be an irreducible representation of depth $r$.
One of the main results in  \cite[Theorem 5.2]{MP94} say that if $V^{P^+_r}\neq 0$ then it contains a  depth-$r$ minimal $K$-type of $P$. Moreover, any two minimal $K$-types
are contained in $V$ are \emph{associated} to each other in the sense of \cite[Section 5.1]{MP94}.

 For any positive integer $r>0$, 
 let $\bar 0\in
 (\ft_r//\rW)^{F}$ be the image of the 
 zero vector $0\in\fg^F_r$
 along the projection $\fg_r^F\to(\ft_r//\rW)^F$
 and 
we denote by 
$(\ft_r//\rW)^{F,\circ}=(\ft_r//\rW)^{F}-\{\bar 0\}$ the complement of $\bar 0$.

\begin{definition}
Let $(P_r,\chi)$
    be a depth-$r$ minimal $K$-type of $G(\mathbbm k)$.
    The \emph{semi-simple part}
of $(P_r,\chi)$, denoted by 
$\theta(\chi)$, is defined as follows -\\
(1) if $r=0$ we define 
$\theta(\chi)\in(\hat\rT//\rW)^F$ to be the image of $\chi\in\mathrm{Irr}(\rG_P^F)$
along the map
\[\mathrm{Irr}(\rG_P^F)\stackrel{\mathcal L_P}\to
(\hat\rT_P//\rW_P)^F\stackrel{t_P}\to(\hat\rT//\rW)^F,\] here $\mathcal L_P$ is the Deligne-Lusztig map for $\rG_P^F$ in~\eqref{DL map}.\\
(2) if $r>0$ we define 
$\theta(\chi)\in(\ft_r//\rW)^{F,\circ}$
to be the image of $\chi\in
(\fg_{P_r}^*)^F$
along the map
\[(\fg_{P_r}^*)^F\cong\fg_{P_r}^F\to(\ft_{P_r}//\rW_P)^F\stackrel{t_{P_r}}\to(\ft_r//\rW)^F
,\] 
where the isomorphism 
$\fg_{P_r}^*\cong\fg_{P_r}$ is induced by  the invariant form
$\langle-,-\rangle_{P_r}$ in~\eqref{invariant form}
and $t_{P_r}$ is the map in~\eqref{t_P_r}.

\end{definition}

   \begin{remark}\label{remark CDT}
       The semi-simple part of minimal $K$-type defined here is same as the Deligne-Lusztig parameter attached to Moy-Prasad types in \cite{chendebackertsai} (Check Section 4.1 for positive depths and 7.1 for depth zero), where it is treated in more generality. 
   \end{remark}

\begin{proposition}\label{DL=K-types}
    Let $(\pi,V)$ be a smooth irreducible representation of depth $r$ and let $(P_r,\chi)$ be a depth-$r$ minimal $K$-type contained in $V^{P^+_r}$.
   We have $\theta(\pi)=\theta(\chi)$.
   In particular, 
   we have $\theta(\pi)\in(\ft_r//\rW)^{F,\circ}$ for $r>0$, and 
   the semi-simple part $\theta(\chi)$ is independent of the choice  of the minimal $K$-type contained in $V$.
    
    \end{proposition}
    \begin{proof}
Pick $0\neq v\in V^{P^+_r}$ such that $\pi^{P_r^+}(g)v=\chi(g)v$ for $g\in P_r/P_r^+$. 
Assume $r=0$. 
For any $z\in\mathbb C[(\hat\rT//\rW)^F]$, we have 
\[z(\theta(\pi))v=
\xi^0(f_z)(v)=\mu(P^+)\sum_{g\in\rG_P^F}j_P(f_z)(g)\chi(g)v\]
where
$f_z\in C^{st}(\rG^F)$ is the image of $z$ under the isomorphism $C[(\hat\rT//\rW)^F]\cong C^{st}(\rG^F)$ in
Theorem \ref{main thm stable group}.
On the other hand, it follows from the defintion of $j_P$ in~\eqref{eqj_P} that 
\[\sum_{g\in\rG_P^F}j_P(f_z)(g)\chi(g)v=\mu(P^+)^{-1}t_P^*(z)(\mathcal L_P(\chi))v=\mu(P^+)^{-1}z(\theta(\chi))v.\]
Thus we have 
$z(\theta(\pi))v=\mu(P^+)\mu(P^+)^{-1}z(\theta(\chi))v$ for any $z\in C[(\hat\rT//\rW)^F]$ and it implies $\theta(\pi)=\theta(\chi)$.

Assume $r>0$. 
For any $z\in\mathbb C[(\ft_r//\rW)^F]$, we have 
\[z(\theta(\pi))v=
\xi^r(f_z)(v)=(\mu(P_r^+)\sum_{g\in\fg_{P_r}^F}j_{P_r}(f_z)(g)\chi(g))v\]
where
$f_z\in C^{st}(\fg_{P_r}^F)$ is the image of $z$ under the isomorphism $\mathbb C[(\ft_r//\rW)^F]\cong C^{st}(\fg_{P_r}^F)$ in
Theorem \ref{main thm stable Lie}.
Note that we have 
\[\mu(P_r^+)\sum_{g\in\fg_{P_r}^F}j_{P_r}(f_z)(g)\chi(g)=
\mu(P_r^+)\sum_{g\in\fg_{P_r}^F}\mathrm{FT}_{P_r}(j_{P_r}(f_z))(g)\mathrm{FT}_{P_r}(\chi^{-1})(g).\]
A direct computation show that 
\[\mathrm{FT}_{P_r}(\chi^{-1})=|\fg_P^F|^{1/2}\mathbbm 1_\chi\]
where $\mathbbm 1_\chi$
is the characteristic function at $\chi\in\fg_{P_r}^F\cong (\fg_{P_r}^*)^F$.
On the other hand, 
it follows from the definition of $j_{P_r}$ in~\eqref{j_{P_r}} that
\[\mathrm{FT}_{P_r}(j_{P_r}(f_z))(g)
=c_{\mu,r}^{-1}\mathrm{FT}^2_{P_r}((\chi_{P_r}^*t_{P_r}^*(z))^-)(g)=
c_{\mu,r}^{-1}((\chi_{P_r}^*t_{P_r}^*(z))^-)(-g)=c_{\mu,r}^{-1}(\chi_{P_r}^*t_{P_r}^*(z))(g).\]
All together, we obtain 
\[z(\theta(\pi))=\mu(P_r^+)\sum_{g\in\fg_{P_r}^F}\mathrm{FT}_{P_r}(j_{P_r}(f_z))(g)\mathrm{FT}_{P_r}(\chi^{-1})(g)=\]
\[=
(\mu(P_r^+)|\fg_P^F|^{1/2}c_{\mu,r}^{-1})\sum_{g\in\fg_{P_r}^F}\chi_{P_r}^*t_{P_r}^*(z)(g)\mathrm 1_\chi(g)=\chi_{P_r}^*t_{P_r}^*(z)(\chi)=z(\theta(\chi)).\]
This implies $\theta(\pi)=\theta(\chi)$.
 \end{proof}

\printbibliography

\bigskip

\textsc{School of Mathematics, University of Minnesota, Vincent Hall, Minneapolis, MN-55455}\\
\textit{E-mail address}: \texttt{bhatt356@umn.edu}

\medskip

\textsc{School of Mathematics, University of Minnesota, Vincent Hall, Minneapolis, MN-55455}\\
\textit{E-mail address}: \texttt{chenth@umn.edu}
\end{document}